\documentclass[11pt,twoside,reqno,centertags]{amsart}
\usepackage{amsmath,amsthm,amsfonts,amssymb}
\advance\textheight by 6truemm
\newtheorem{Theorem}{Theorem}
\newtheorem{theorem}[Theorem]{Theorem}
\newtheorem{corollary}{Corollary}

\newtheorem{lemma}{Lemma}
\newtheorem{remark}{Remark}
\newtheorem{definition}{Definition}
\newtheorem{example}{Example}

\usepackage{fancybox,color}

\newcommand{\be}{\begin{equation} \label}
\newcommand{\ee}{\end{equation}}
\newcommand{\R}{{\mathbb R}}
\newcommand{\eps}{\varepsilon}

\numberwithin{equation}{section}

\begin{document}

\title[Threshold solutions]{Threshold, subthreshold and global unbounded \\
solutions of superlinear heat equations}
\author[Pavol Quittner and Philippe Souplet]{Pavol Quittner$^{(1)}$ and Philippe Souplet$^{(2)}$}

\thanks{$^{(1)}$Department of Applied Mathematics and Statistics, Comenius University
Mlynsk\'a dolina, 84248 Bratislava, Slovakia. Email: quittner@fmph.uniba.sk} 

\thanks{$^{(2)}$Universit\'e Sorbonne Paris Nord, CNRS UMR 7539, LAGA,
93430 Villetaneuse, France. Email: souplet@math.univ-paris13.fr}

\date{}

\begin{abstract}
We consider the semilinear heat equation with a superlinear nonlinearity
and we study the properties of threshold or subthreshold
solutions, lying on or below the boundary between blow-up and global
existence, respectively. 
For the Cauchy-Dirichlet problem, we prove the boundedness and decay to zero of any subthreshold solution.
This implies, in particular, that all global unbounded solutions 
--- if they exist --- are threshold solutions.
For the Cauchy problem, these properties fail in general but we show that they become
true for a suitably modified notion of threshold.
Our results strongly improve known results even in the model case of power nonlinearities,
especially in the Sobolev critical and supercritical cases.



\end{abstract}

\maketitle

\section{Introduction}
\label{sec-intro}

We consider nonnegative classical solutions of the problem
\begin{equation} \label{MPf}
\begin{aligned}
u_t-\Delta u &= f(u) &\quad&\hbox{in }\ \Omega\times(0,T), \\
u &=0 &\quad&\hbox{on }\partial\Omega\times(0,T), \\
 u(\cdot,0) &= u_0 &\quad&\hbox{in }\ \Omega, 
\end{aligned}
\end{equation}
where 
\begin{equation} \label{ass-Omegap0}
\hbox{either }\ \Omega\subset\R^n \hbox{ is bounded and smooth or }\Omega=\R^n,
\end{equation}
\begin{equation} \label{ass-u0intro}
u_0\in C(\bar\Omega)\cap L^\infty(\Omega), \quad u_0\ge0,
\end{equation}
and $T=T(u_0)\leq\infty$ is the maximal existence time.
Of course, the boundary condition is not present if $\Omega=\R^n$.
We will assume that either the nonlinearity $f$ satisfies 
\begin{equation} \label{f10}
f\in C^1([0,\infty))\cap  C^2((0,\infty)),\quad\hbox{$f\ge 0$,  \quad $f(0)=0$,\quad $f$ convex,}
\quad \int^\infty\frac{ds}{f(s)}<\infty,
\end{equation}
or $f\in C^1$ satifies conditions \eqref{f1}--\eqref{f4} below
(those conditions require, among other things, a Sobolev-critical growth of $f$,
but do not require $f$ to be convex).

Our assumptions guarantee that the solution $u$ of
\eqref{MPf} with $u_0$ suitably large blows up in finite time
(as is well known, if $f$ is convex and nonnegative at infinity, then the integral condition in \eqref{f10}
is in fact necessary and sufficient for this to happen).
On the other hand, if $\Omega$ is bounded and $f'(0)$ is small enough or if $\Omega=\R^n$
and $f(u)\leq Cu^q$ for some $q>1+2/n$ and $u$ small,
then the solution $u$ of
\eqref{MPf} with $u_0$ suitably small is global and tends to zero as
$t\to\infty$.
We will be mainly interested in the behavior of 
solutions lying on or below the borderline between global existence and blow-up.
To be more precise, we need the following definition\footnote{The notion 
``$\eps f$-(sub)threshold'' is new, while
``$\eps$-(sub)thresholds'' or ``thresholds'' are sometimes called ``(sub)thresholds''
or ``strong thresholds'', respectively, see \cite{Q17},
for example.}.

\begin{definition} \label{defThr} 
Let $u^*$ be a solution of \eqref{MPf}
with initial data $u_0^*\geq0$. We say that $u$ is:
\vskip 3pt

$\bullet${\hskip 3pt}an {\bf $\eps$-threshold} if, for any $\eps>0$,
we have $T(u_0)<\infty$ whevener $u_0\ge(1+\eps)u_0^*$, and
$T(u_0)=\infty$ whevener $u_0^*\ge(1+\eps)u_0$.

\vskip 3pt

$\bullet${\hskip 3pt}an {\bf $\eps f$-threshold} if, for any $\eps>0$,
we have $T(u_0)<\infty$ whevener $u_0\ge u_0^*+\eps f(u_0^*)$, and
$T(u_0)=\infty$ whevener $u^*_0\ge u_0+\eps f(u_0)$.

\vskip 3pt

$\bullet${\hskip 3pt}a {\bf threshold} if, for any $u_0\not\equiv u_0^*$,
we have $T(u_0)<\infty$ whevener $u_0\ge u_0^*$, and
$T(u_0)=\infty$ whevener $u^*_0\ge u_0$.

\smallskip
Let $u^*$ be an $\eps$-threshold with initial data $u_0^*$.
If $u_0$ satisfies $(1+\eps)u_0\le u_0^*$ for some $\eps>0$ 
or $u_0+\eps f(u_0)\le u_0^*$ for some $\eps>0$
or $u_0\le u_0^*$ and $u_0\not\equiv u_0^*$,
then the solution $u$ with initial data $u_0$ 
is called an {\bf $\eps$-subthreshold} or {\bf $\eps f$-subthreshold} or
{\bf subthreshold}, respectively.
\end{definition}

Since we consider only bounded initial data, our assumptions on $f$
guarantee that\footnote{assuming $f(u_0^*)\not\equiv0$ for the first implication} 
$$\hbox{$u^*$ is a threshold} \Rightarrow \hbox{$u^*$  is an $\eps f$-threshold} \Rightarrow \hbox{$u^*$ is an $\eps$-threshold.}$$
It is also easy to see that these three notions coincide if $\Omega$ is bounded.
On the other hand, if $\Omega=\R^n$, 
then we will show (see Remark~\ref{rem-epsf}) that $\eps$-thresholds need not be $\eps f$-thresholds
and $\eps f$-thresholds need not be thresholds.
It is also known that $\eps$-thresholds
can be global and bounded (and even converge to zero), global unbounded,
and they can also blow up in finite time.
Here and in what follows, {\bf bounded} or {\bf unbounded} solution means 
$\sup_{t>0}\|u(\cdot,t)\|_\infty<\infty$
or
$\limsup_{t\to\infty}\|u(\cdot,t)\|_\infty=\infty$, respectively.
Definition~\ref{defThr} also implies that\footnote{assuming $f(u_0^*)\not\equiv0$ for the second implication}
$$\hbox{$u$ is an $\eps$-subthreshold} \Rightarrow \hbox{$u$  is an $\eps f$-subthreshold} \Rightarrow \hbox{$u$ is a subthreshold.}$$
It is known that in some cases, some or all 
subthresholds are bounded
or even tend to zero as $t\to\infty$,
but these properties are not true in general (see below).

We will be interested in sufficient conditions guaranteeing
positive answers to the following questions:

\vskip 1pt

(Q1) Are subthresholds bounded~?

\vskip 1pt

(Q2) Do subthresholds converge to zero as $t\to\infty$~?

\vskip 1pt

(Q3) Are $\eps$-thresholds thresholds~?

\vskip 1pt

(Q4) Are global unbounded solutions thresholds~?

\smallskip

In order to discuss related known results,
let us consider the model case $f(u)=u^p$ with $p>1$,
and let us also denote 
$$ 
 p_S :=\begin{cases}\ +\infty       & \hbox{ if }n\leq2, \\
                       1+\frac4{n-2} & \hbox{ if }n\geq3,
           \end{cases} \qquad
 p_{JL} :=\begin{cases}\ +\infty & \hbox{ if }n\leq10, \\
         1+4\frac{n-4+2\sqrt{n-1}}{(n-2)(n-10)} & \hbox{ if }n>10,
          \end{cases}
$$ 
the critical Sobolev and Joseph-Lundgren exponents.

First consider the case $p<p_S$. 
Then all global solutions are bounded
and the bound depends only on the bound of the initial data
(see \cite{G86} or \cite{Q21} if $\Omega$ is bounded or $\Omega=\R^n$, respectively). 
This implies that all $\eps$-thresholds are global and bounded.
If $\Omega=\R^n$, then all $\eps$-thresholds (hence also all subthresholds)
even tend to zero as $t\to\infty$ due to \cite{Q21}.
If $\Omega$ is bounded, then the nonexistence of multiple 
ordered positive steady states
and the comparison principle 
also imply that subthresholds tend to zero
(but this is not true for thresholds). 
See also \cite{Q03} and \cite{QS25p} for similar results in the case 
of more general subcritical nonlinearities $f$.

Next consider the case $p>p_S$. 
First assume that $\Omega$ is convex or $\Omega=\R^n$ and $u_0\in C^1$ satisfies 
$$u_0(x)+|x||\nabla u_0(x)|=o(|x|^{-2/(p-1)})\quad\hbox{as }\ |x|\to\infty.$$
If the solution $u$ is global, then it is bounded and decays to zero as $t\to\infty$, 
but the bound depends on $u_0$ and not just on $\|u_0\|_\infty$
(see \cite{CDZ,BS,S17} for $\Omega$ bounded; for $\Omega=\R^n$, see \cite{S17}, where different decay assumptions can also be found,
and see also \cite{MM09} for an earlier result in the radial case).
If $u$ is an $\eps$-threshold, then $T(u_0)<\infty$
(and the blow-up is often of type II, provided $p\ge p_{JL}$, see \cite{QS19} and the references therein).
On the other hand, if $\Omega=\R^n$ and the decay of $u_0$ is not fast enough,
then $u$ can be a global unbounded $\eps$-threshold
and it can also exhibit interesting asymptotic behavior,
for example
$$ 0=\liminf_{t\to\infty}\|u(\cdot,t)\|_\infty<
  \limsup_{t\to\infty}\|u(\cdot,t)\|_\infty=\infty,$$
see \cite{PY03,PY14,Q17}.
If $p\ge p_{JL}$, then there exist global unbounded radial $\eps$-thresholds 
which are not thresholds and some subthresholds can also be global unbounded,
see Remark~\ref{rem-pJL} for more information on their properties.
Let us also mention that if $p\in(p_S,p_{JL})$ and $\Omega$ is a ball,
then the boundedness of radial global solutions was proved in \cite{CFG08}
for a more general $f$ satisfying $\lim_{u\to\infty}u^{-p}f(u)=1$,
see also \cite{FK24} for a recent, more general result.

If $p\ge p_S$ and $\Omega$ is bounded and starshaped, then the convergence to zero of bounded solutions
follows from the compactness of the semiflow, existence of a Lyapunov functional,
and the nonexistence of positive steady states.
However, if $\Omega=\R^n$ and $p_S<p<p_{JL}$, for example,
then there exist bounded solutions which satisfy
$$ 0=\liminf_{t\to\infty}\|u(\cdot,t)\|_\infty<
  \limsup_{t\to\infty}\|u(\cdot,t)\|_\infty<\infty,$$
and their $\omega$-limit sets do not contain any positive steady state, see \cite{PY14s}.

Finally consider the critical case $p=p_S$.
If $\Omega$ is a ball and $u_0$ is radially symmetric and radially nonincreasing,
then the existence of global unbounded thresholds follows from
\cite{GaV97} (see also \cite{GK03} and \cite[Theorem 22.9]{QS19}).
If $\Omega$ is bounded (not necessarily a ball),
then the existence of global unbounded solutions has been
proved in \cite{CdPM20} and \cite{AdP25} for $n\ge5$ and $n=3$, respectively.
If $\Omega=\R^n$ and $n=3$ or $n=4$, then the existence 
of global unbounded ($\eps$-threshold, radially symmetric) solutions 
has been predicted in \cite{FK12} and the corresponding rigorous results have been
obtained in \cite{dPMW20} or \cite{WZZ24}, respectively.
On the other hand, as far as we know, if $u$ is a global unbounded $\eps$-threshold,
then the boundedness of subthresholds has not been studied in an explicit way.

Our results improve or generalize known results in several directions.
In the case of nonlinearities
satisfying \eqref{f10}, 
our main results
are summarized in the following theorems.

\goodbreak
\eject

\begin{theorem} \label{thm-introA}
Let $\Omega$ be bounded and smooth and assume \eqref{ass-u0intro}-\eqref{f10}.

\begingroup
    \addtolength{\leftmargini}{-10pt}
\begin{itemize}
\item[(i)] 
Any subthreshold is bounded.  
If in addition $f$ is strictly convex, then any subthreshold $u$ satisfies $\lim_{t\to\infty}\|u(\cdot,t)\|_\infty=0$.

\smallskip
\item[(ii)] 
Any global unbounded solution is a threshold.
\end{itemize}
\endgroup
\end{theorem} 

\goodbreak

\begin{theorem} \label{thm-intro}
Let $\Omega=\R^n$ and assume \eqref{ass-u0intro}-\eqref{f10}. 

\begingroup
    \addtolength{\leftmargini}{-10pt}
    \begin{itemize}
\item[(i)] If $u^*$ is an $\eps$-threshold, then each $\eps f$-subthreshold is bounded.  
\smallskip

\item[(ii)] Any global unbounded solution is an $\eps f$-threshold.
\smallskip

\item[(iii)] Let $u^*$ be an $\eps f$-threshold, with initial data $u_0^*$ such that
\begin{equation} \label{assL1}
\hbox{ $u_0^*$ is radial, radially nonincreasing and $f(u_0^*)\in L^1(\R^n)$}.
\end{equation}
Then $u^*$ is a threshold and each subthreshold is bounded.
\end{itemize}
\endgroup
\end{theorem} 

\goodbreak

We also obtain similar assertions for $f$ satisfying \eqref{f1}--\eqref{f4}
(see Section~\ref{sec-3})
and provide sufficient conditions for the convergence of subthresholds to zero
(see Corollary~\ref{cor-nonrad} and Theorem~\ref{thm-nonrad-3}).

\vskip 2pt
Our results are new even for the model case $f(u)=u^p$ with $p>1$.
For example:

\vskip 2pt

$\bullet$ If $p>p_S$ and $\Omega$ is bounded, then the boundedness and convergence to zero
of all subthresholds is known for convex domains (see \cite{CDZ,BS,S17}),
but our results guarantee these properties without convexity assumption on the domain.

\vskip 2pt

$\bullet$ Similarly, if $p=p_S$, $\Omega$ is a general bounded domain, and $u$ is
a global
unbounded solution, then the boundedness and the convergence to zero
of all solutions lying below $u$ 
and the blow-up of all solutions lying above $u$ were also unknown.

\vskip 2pt

$\bullet$ If $\Omega=\R^n$, $n\in\{3,4\}$, and $p=p_S$,
then radially nonincreasing initial data $u_0$ of known global unbounded solutions  
(cf.~\cite{dPMW20,WZZ24}) satisfy 
$\lim_{|x|\to\infty}{|x|^\gamma}u_0(x)=A>0$ for some $A>0$ and $\gamma>n-2$, hence $f(u_0)\in L^1(\R^n)$,
and Theorem~\ref{thm-intro}(iii)
shows that such solutions 
are thresholds and all subthresholds are bounded. 

\vskip 2pt

$\bullet$ Let $\Omega=\R^n$, $p>p_{JL}$ and let $U_*(x)=c_p|x|^{-2/(p-1)}$ denote the singular 
stationary solution. 
For $\alpha>0$, let $S_\alpha$ be the set of solutions $u$
with initial data $u_0$ lying below $U_*$ and satisfying
$u_0(x)=U_*(x)-\eps|x|^{-\alpha}$ for some $\eps>0$ and $|x|$ large.
Then all solutions in $S_\alpha$ are global unbounded whenever $\alpha$ is sufficiently large,
and the value $\ell=\ell(n,p)>0$ of the infimum of all $\alpha$ with this property is known
from  \cite{PY03,FKWY06,FKWY07,FKWY11} (see Remark~\ref{rem-pJL}).
On the other hand, if $\tilde S_\beta$ is the set of
solutions $u$ with radially nonincreasing initial data $u_0$ satisfying
$u_0(x)=U_*(x)+\eps|x|^{-\beta}$ for some $\eps>0$ and $|x|$ large,
and
$\tilde\ell$ is the infimum of all $\beta>0$ such that $\tilde S_\beta$ contains a global solution,
then it is well known that $\tilde\ell\ge2/(p-1)$ and
\cite{Q17} implies $\tilde\ell\le n$,
but our results guarantee a more precise estimate $\tilde\ell\le 2p/(p-1)$.
In addition, our estimate of $\tilde\ell$ (as well as our estimate of $\ell$)
remain true for more general
initial data (and nonlinearities $f$), see Remark~\ref{rem-pJL}.

\vskip 2pt

We stress that 
the class of nonlinearities $f$ satisfying \eqref{f10}
is very large: it also contains slightly superlinear functions like
$f(u)=u\log^2(1+u)$ as well as exponentially growing functions like
$f(u)=e^u-u-1$, for example.
Further examples and applications of our results can be found in Section~\ref{sec-ex}.

\vskip 2pt
In Section~\ref{sec-sub} we state additional results for nonlinearities
satisfying assumptions \eqref{f10}, and prove them along with Theorems~\ref{thm-introA} and~\ref{thm-intro}.
In Section~\ref{sec-3} we state and prove our results for nonlinearities
satisfying assumptions \eqref{f1}--\eqref{f4}.
The arguments in the proofs are very different.
While in Section~\ref{sec-sub} we use bootstrap arguments and estimates in weighted or uniformly local
Lebesgue spaces along with subsolutions of the form $u+\eps f(u)$ (requiring $f$ to be convex),
in Section~\ref{sec-3} we use subsolutions of the form $(1+\eps)u$ and energy arguments,
but also rescaling and intersection properties of entire positive solutions 
of the equation $-\Delta u=u^{p_S}$ (and our arguments, among other things, 
require $f(u)$ to behave like $u^{p_S}$ for large $u$, see \eqref{f2}). 

\goodbreak

\section{Convex nonlinearities} 
\label{sec-sub}

Theorems~\ref{thm-introA} and~\ref{thm-intro} will be consequences of the following theorem,
which establishes the boundedness of solutions starting below a global solution,
and of its corollary. The latter also provides sufficient conditions for the convergence of subthresholds to zero.

\begin{theorem} \label{thm-subsol2}
Assume \eqref{ass-Omegap0}, \eqref{f10}, 
and let $u_0, v_0$ satisfy \eqref{ass-u0intro}
with $u_0\ge v_0$.
Denote by $u, v$ the corresponding solutions of \eqref{MPf}
and assume that $u$ is global.

\smallskip
\begingroup
    \addtolength{\leftmargini}{-10pt}
\begin{itemize}
\item[(i)]
If $\Omega$ is bounded and $v_0\not\equiv u_0$, then $v$ is bounded.
If in addition $f$ is strictly convex, then $\lim_{t\to\infty}\|v(\cdot,t)\|_\infty=0$.

\smallskip
\item[(ii)] Assume $\Omega=\R^n$. If $u_0\ge (1+\eps)v_0$ for some $\eps>0$,
then $v$ is bounded.
The conclusion is more generally true whenever
 \be{hypepsf}
 \hbox{$u_0\ge v_0+\eps f(v_0)$ for some $\eps>0$.}
 \ee
 
 \smallskip
\item[(iii)] Under assumption \eqref{hypepsf}, $v$ satisfies 
the a priori estimate
 \be{AEepsf}
 \sup_{t\ge 0}\|v(t)\|_{L^\infty(\Omega)} \le C(\eps,M,\Omega,f)\quad\hbox{if $\|v_0\|_\infty\le M$.}
 \ee
 \end{itemize}
 \endgroup
\end{theorem}

\begin{corollary} \label{cor-nonrad}
Assume \eqref{ass-Omegap0}--\eqref{f10}. 
Denote by $u$ the corresponding solution of \eqref{MPf}
and assume that $u$ is global and unbounded.
If $\Omega=\R^n$, then assume also that $u_0$ is radial, radially
nonincreasing and $f(u_0)\in L^1(\R^n)$.
Then $u$ is a threshold.
More precisely, let $v$ be the solution with initial data $v_0$,
where $v_0$ satisfies \eqref{ass-u0intro}
and $v_0\not\equiv u_0$. 
Then we have the following:

\smallskip
\begingroup
    \addtolength{\leftmargini}{-10pt}
\begin{itemize}
\item[(i)] If $v_0\ge u_0$, then $v$ blows up in finite time.

\smallskip

\item[(ii)] If $v_0\le u_0$, then $v$ is global and bounded.

\smallskip

\item[(iii)]
Assume $\Omega=\R^n$, $f(u)=u^{p_S}$. Denote $r:=|x|$ and 
assume $u_0=u_0(r)$ is piecewise $C^1$ and 
\begin{equation} \label{u0-supp1}
\hbox{either \ $u_0(r)=0$ for
$r\ge R_0$, \ \ \ $u_0'(r)<0$ for $r\in[R_0-\eta,R_0]$,}
\end{equation}
\begin{equation} \label{u0-supp2}
\hbox{or \ 
$\lim\limits_{r\to\infty} r^\gamma u_0(r) 
 = -\lim\limits_{r\to\infty}\frac1\gamma  r^{\gamma+1} u_0'(r)=C_0>0$,
\ \ $\gamma>n-2$,}
\end{equation}
where $R_0>\eta>0$.
If $v_0\le u_0$, then $\lim_{t\to\infty}\|v(\cdot,t)\|_\infty=0$. 
 \end{itemize}
 \endgroup
\end{corollary}

\begin{remark} \label{rem-decay} \rm
(i) Let $p=p_S$, $\Omega=\R^n$, $n\in\{3,4\}$.
 \cite[Conjecture 1.1]{FK12} predicts that 
if $u$ is an $\eps$-threshold with positive radial initial data satisfying 
\begin{equation} \label{decayFK}
\hbox{$\lim_{|x|\to\infty}|x|^\gamma u_0(x)=A>0$, where $\gamma>n-2$,}
\end{equation}
 then $u$ is global unbounded. 
Notice that the decay condition \eqref{decayFK} corresponds to the decay in
\eqref{u0-supp2}.
Global unbounded positive radial solutions satisfying \eqref{decayFK}
have been constructed in \cite{dPMW20,WZZ24}
and this result combined with arguments in
\cite[Remark~2(i)]{Q17} shows that global unbounded solutions with radially
nonincreasing initial
data satisfying assumptions \eqref{u0-supp1} or \eqref{u0-supp2} exist.
We can actually show that:
$$\left. \begin{aligned}
&\hbox{If $u$ is an $\eps$-threshold with radial nonincreasing, piecewise $C^1$ initial data }\\
&\hbox{satisfying \eqref{u0-supp1}, then $u$ is  global and unbounded.}
\end{aligned}\ \right\} 
$$
First note that \cite{GaV97} implies $u$ is global.
 To show that $u$ is unbounded,  we take
a global unbounded solution $v$ with positive radial initial data $v_0\in C^1$,
and let $w$ be the (global unbounded) rescaled solution with initial data $w_0(x):=\eta v_0(\eta^{(p_S-1)/2}x)$.
The arguments in \cite[Remark~2]{Q17} guarantee that we may assume that $v_0$ (hence also $w_0$) is radially nonincreasing. 
If $\eta>0$ is small enough, then $u_0(0)>w_0(0)$ and $z(u_0-w_0)=1$,
where $z(\cdot)$ denotes the zero number in $(0,\infty)$
(see \cite[Subsection~52.8]{QS19} for the definition and properties of the zero number).
 We claim that $u(0,t)\ge w(0,t)$ for all $t$  (so that $u$ is unbounded).
Assume to the contrary that $u(0,t_0)<w(0,t_0)$ for some $t_0>0$.
Then we can find $\lambda>1$ such that the solution $u_\lambda$ with initial data $\lambda u_0$
satisfies $\lambda u_0(0)>w_0(0)$, $z(\lambda u_0-w_0)=1$ and $u_\lambda(0,t_0)<w(0,t_0)$.
This implies $u_\lambda\le w$  for all $t\in[t_0,T(\lambda u_0))$, hence $u_\lambda$ is global, which yields a contradiction.

\smallskip
 
(ii) The assumption $f(0)=0$ in Theorems~\ref{thm-introA} and \ref{thm-subsol2} cannot be replaced
by $f(0)\ge 0$. Indeed consider $f(u)=(u+a)^p$ with $a>0$, $n\ge 11$ and $p>p_{JL}$.
There exists $R=R(a)>0$ such that
problem \eqref{MPf} with $\Omega=B_R$ admits infinitely many ordered global unbounded solutions.
In fact any nonnegative solution starting below the singular steady
state $U_*-a$ is global unbounded; see \cite{DGLV}.
A similar phenomenon occurs for $f(u)=e^u$ with $n\ge 11$ and suitable $R>0$.

\smallskip

(iii) The assumption $f\ge 0$ is used in an essential way in our proofs and
we thus do not cover nonlinearities such that $f'(0)<0$ 
(including the so-called bistable nonlinearities).
Related results on the behavior of threshold solutions and their instability properties for such nonlinearities can be found
in \cite{Pol11} 
(see also \cite{Zl06, DM10} for earlier results for one-dimensional problems). 
These works are based on rather different techniques which make essential use of the assumption $f'(0)<0$
and in particular do not apply to the model case $f(u)=u^p$. Also the results in \cite{Pol11} often require Sobolev subcritical growth of $f$ at infinity.

\smallskip

(iv) The strict convexity assumption on $f$ in the second part of Theorem~\ref{thm-subsol2}(i) cannot be replaced by mere convexity.
Indeed, let $\Omega$ be any smooth bounded domain, fix any $\eta>0$, $p>1$, denote by $\lambda_1$ 
the first eigenvalue of $-\Delta$ in $\Omega$ with zero Dirichlet conditions, and by $\varphi_1$ the corresponding eigenfunction
such that $\max_{\overline\Omega}\varphi_1=1$.
It is easily seen that there exist convex functions $f\in C^2([0,\infty))$ such that $f(s)=\lambda_1 s$ for $s\in[0,\eta]$ 
and $f(s)=s^p$ for $s$ large.
Then assumption \eqref{f10} is satisfied, but \eqref{MPf} has infinitely many ordered positive steady states, given by $\mu\varphi_1$ for any $\mu\in(0,\eta]$.

\smallskip

(v) Theorems~\ref{thm-introA} and \ref{thm-subsol2} remains valid for problem \eqref{MPf} with the Neumann boundary conditions $\partial_\nu u=0$
instead of the Dirichlet boundary conditions $u=0$.
This follows from the proof below upon replacing the weighted Lebesgue spaces $L^q_\delta$ by the usual 
 Lebesgue spaces~$L^q$.
\qed
\end{remark}

\smallskip

\goodbreak

 The proof of Theorem~\ref{thm-subsol2} is based on the following three lemmas.
Our first lemma contains a key comparison argument for suitably ordered initial data
(valid for general convex nonlinearities such that $f(0)=0$).

\begin{lemma} \label{lem-subsol}
Assume \eqref{ass-Omegap0} and let $f\in C^1([0,\infty))\cap C^2((0,\infty))$, with $f(0)=0$ and $f$ convex.
Let $u_0, v_0$ satisfy \eqref{ass-u0intro} with
$u_0\ge v_0+\eps f(v_0)$ for some $\eps>0$.
Denote by $u, v\ge 0$ the corresponding solutions of \eqref{MPf}.
Then we have $T(v_0)\ge T(u_0)$ and 
\be{compvu}
u\ge v+\eps f(v) \quad\hbox{in $\Omega\times (0,T(u_0))$.}
\ee
\end{lemma}

\begin{proof}
We may assume without loss of generality that $v_0\not\equiv 0$
(since otherwise $v\equiv 0$ and the conclusion is trivial).
We thus have $v>0$ in $\Omega\times(0,T(v_0))$ by the strong maximum principle.
We compute 
$$\partial_t(f(v))=f'(v)v_t,\quad \Delta(f(v))=f'(v)\Delta v+f''(v)|\nabla v|^2.$$
Since $f''\ge 0$ on $(0,\infty)$, the function $z:=v+\eps f(v)$ satisfies
$$z_t-\Delta z\le (1+\eps f'(v))(v_t-\Delta v)=(1+\eps f'(v))f(v)$$
and, by Taylor's formula, we have
$$f\bigl(v+\eps f(v)\bigr)=f(v)+f'(v)\eps f(v)+\frac{f''(\theta v)}{2}(\eps f(v))^2\ge f(v)+f'(v)\eps f(v)=(1+\eps f'(v))f(v),$$
for some function $\theta=\theta(x,t)\in [0,1]$.
Consequently,
$$z_t-\Delta z\le f(z) \quad\hbox{in $\Omega\times (0,T(v_0))$.}$$
Let $T:=\min(T(u_0),T(v_0))$.
If $\Omega$ is bounded, we moreover have $u=z=0$ on $\partial\Omega\times(0,T)$,
owing to $f(0)=0$.
Since $u_0\ge z(\cdot,0)$ in $\Omega$ by assumption, 
it follows from the comparison principle (see, e.g., \cite[Proposition 52.6 or 52.10]{QS19}, 
which apply for $\Omega=\R^n$ as well as $\Omega$ bounded) that
$u\ge z$ in $\Omega\times (0,T)$, i.e.~\eqref{compvu} holds with $T(u_0)$ replaced by $T$.
Finally, since $f$ is convex and $f(0)=0$, we either have $f\le 0$ on $[0,\infty)$ or $f$ bounded below. 
In the first case, we immediately get $T(u_0)=T(v_0)=\infty$. In the second case, setting $K:=\sup_{[0,\infty)}(-f)\in [0,\infty)$, we obtain
$v\le u-\eps f(v)\le u+\eps K$ in $\Omega\times (0,T)$, hence $T(v_0)\ge T(u_0)$.
The proof of the lemma is complete.
\end{proof}

The second lemma provides a (universal) a priori estimate 
in $L^1_\delta$ or $L^1_{ul}$ spaces for global solutions of \eqref{MPf}. 
We first recall the definition of the spaces $L^q_\delta$ or $L^q_{ul}$ for $q\in[1,\infty]$:
$$L^q_\delta(\Omega)=L^q(\Omega,\delta(x)dx),$$
where $\delta(x)={\rm dist}(x,\partial\Omega)$ is the distance to the boundary, and
$$L^q_{ul}=L^q_{ul}(\R^n)=\bigl\{w\in L^q_{loc}(\R^n);\ \sup_{a\in\R^n}\|w\|_{L^q(B_1(a))}<\infty\bigr\},$$
both equipped with their natural norms.
We note that $L^\infty_\delta(\Omega)=L^\infty(\Omega)$, with same norms, as well as $L^\infty_{ul}(\R^n)=L^\infty(\R^n)$.
In the rest of this section, for convenience, we will denote
$$
X_q=
\begin{cases}
L^q_\delta(\Omega)& \hbox{ if $\Omega$ is bounded,} \\ 
\noalign{\vskip 1mm}
L^q_{ul} & \hbox{ if $\Omega=\R^n$.} 
\end{cases}$$

\begin{lemma} \label{lem-UB}
Assume \eqref{ass-u0intro}, \eqref{f10} and $T(u_0)=\infty$.
Then
$$\sup_{t\ge 0} \|u(t)\|_{X_1} \le M,$$
for some constant $M=M(f,\Omega)>0$.
\end{lemma}

The a priori bounds in Lemma~\ref{lem-UB}, based on Kaplan's eigenfunction method, are essentially known.
For the case $f(s)=s^p$, see \cite[p.~103]{NST} or \cite[p.~109]{FSW} when $\Omega$ bounded, 
and \cite[Lemma~26.16]{QS19} (see also \cite[Lemma~4.4]{MaSou}) when $\Omega=\R^n$.
We give the short proof for general $f$ for convenience.

\begin{proof}
In this proof, the symbols $C_i$ denote positive constants depending only on $f,\Omega$.
\smallskip

First consider the case when $\Omega$ is bounded.
Let $\varphi_1$ be the first positive eigenfunction of $-\Delta$ in $\Omega$ with zero Dirichlet conditions, 
normalized by $\int_\Omega\varphi_1=1$, and denote by $\lambda_1$ the corresponding eigenvalue.
Set $y(t)=\int_\Omega u(t)\varphi_1$.
By \eqref{f10}, we have $\lim_{s\to\infty} f(s)/s=\infty$, hence $\frac12 f(s)\ge \lambda_1s-C_1$ for all $s\ge 0$.
Multiplying the PDE in \eqref{MPf}  by $\varphi_1$, integrating by parts and using Jensen's inequality, it follows that
\be{diffineqy}
y'(t)\ge \int_\Omega f(u(t))\varphi_1 
-\lambda_1 \int_\Omega u(t)\varphi_1
\ge f(y(t))-\lambda_1 y(t)
 \ge  \frac12 f(y(t))-C_1.
 \ee
Since $u$ is global, $\lim_{s\to\infty} f(s)=\infty$, and $\int^\infty\frac{ds}{f(s)}<\infty$, 
we deduce easily that $y(t)\le C_2$ for all $t\ge 0$.
Since $\varphi_1\ge C_3\delta$ by Hopf's lemma, assertion (i) follows.

\smallskip

Now consider the case when $\Omega=\R^n$.
Let $\varphi_1, \lambda_1, y(t)$ be as in the previous paragraph
with $\Omega$ replaced by $B_2$.
Arguing as above, also using $\partial_\nu\varphi_1\leq 0$ on $\partial B_2$, we again obtain \eqref{diffineqy},
hence $y(t)\le C_4$ for all $t\ge 0$.
Consequently, 
$$\int_{B_1} u(t)\,dx\leq C(n)\int_{B_2} u(t)\varphi_1\,dx\le C_5,\quad t\ge 0.$$
The conclusion then follows by applying this to $u(x-a,t)$ and taking 
the supremum over $a\in\R^n$.
 \end{proof}

Our third lemma is a 
smoothing estimate in $L^q_\delta$ or $L^q_{ul}$ spaces for the following 
{\it sublinear},\footnote{We point out that similar results for the, 
more familiar, linear problem with $u$ instead of $|u|^\theta$ would not be sufficient for our needs,
in order to apply to arbitrary nonlinearities satisying \eqref{f10}.}
auxiliary parabolic equation with a potential:
\begin{equation} \label{MPflin}
\begin{aligned}
v_t-\Delta v &= W(x,t)|v|^\theta&\quad&\hbox{in }\ \Omega\times(0,T), \\
v &=0 &\quad&\hbox{on }\partial\Omega\times(0,T), 
\end{aligned}
\end{equation}
with $\theta\in [0,1)$.

\begin{lemma} \label{lem-smoothing}
Set $m_0=(n+1)/2$ if $\Omega$ is bounded and $m_0=n/2$ if $\Omega=\R^n$.
Let $T\in(0,\infty]$, $m\in[1,\infty)$ with $m>m_0$, $M>0$, $\tau\in(0,1]$ 
and $W\in L^\infty(0,T;X_m)$.
There exists $M_1>0$, depending only on $M,\tau,m,\Omega$ (independent of $T$) such that,
if $v\in L^\infty(0,T;L^\infty(\Omega))$ is a strong (or more generally, mild)
solution of \eqref{MPflin} with $\theta=1-\frac{1}{m}$ and
$$\|v\|_{L^\infty(0,T;X_1)} + \|W\|_{L^\infty(0,T;X_m)}\le M,$$
then
\be{condpq3}
\|v\|_{L^\infty(\tau,T;L^\infty(\Omega))} \le M_1.
\ee
\end{lemma}

\begin{proof}
We may assume $T<\infty$.
Denote by $e^{t\Delta}$ the heat semigroup on $L^\infty(\Omega)\subset X_p$, 
with Dirichlet conditions in the case $\Omega$ bounded.
We know that, for all $1\le p\le q\le \infty$,
\be{XpXq}
\|e^{t\Delta}\|_{\mathcal{L}(X_p,X_q)}\le C(\Omega) t^{-\gamma},\quad t>0,\quad\hbox{where }  
\gamma=m_0\Bigl(\frac{1}{p}-\frac{1}{q}\Bigr)
\ee
(see \cite{FSW} for $\Omega$ bounded and \cite{GV97} for $\Omega=\R^n$).
Let $\tau\in (0,1]$ with $\tau<T$ and let $t\in(0,T-\tau)$. Since $v$ is a (mild) solution of \eqref{MPflin}, it satisfies
\be{varconstu}
v(t+\tau)=e^{\tau\Delta}v(t)+\int_0^\tau e^{(\tau-s)\Delta}h(s)\, ds,
\quad h:=W|v|^{1-\frac{1}{m}}.
\ee
Let $1\le p<\infty$, let $q\in(p,\infty]$ satisfy
\be{condpq}
\frac{1}{p}-\frac{1}{q}<\frac{1}{m_0}-\frac{1}{m}
\ee
and set
\be{condpqr}
r=\frac{mp}{m+p-1}\in [1,\min(m,p)].
\ee
For $m>1$, we have $r<m$ and $\frac{r(m-1)}{m}(\frac{m}{r})'=\frac{r(m-1)}{m-r}=p$ hence,
by H\"older's inequality,
$$\bigl\||h|^r\bigr\|_{X_1}\le \bigl\||W|^r\bigr\|_{X_{m/r}}\bigl\||v|^{r(m-1)/m}\bigr\|_{X_{(m/r)'}}
= \|W\|^r_{X_m} \|v\|^{r(m-1)/m}_{X_p}.$$
Consequently (treating the case $m=1=r$ separately), we obtain
\be{condpqh}
\|h(s)\|_{X_r}\le M\|v(s)\|_{X_p}^{(m-1)/m},\quad s\in(0,T).
\ee
Also, \eqref{condpq}, \eqref{condpqr} guarantee that
$$\nu:=m_0\Bigl(\frac{1}{r}-\frac{1}{q}\Bigr)
<m_0\Bigl(\frac{1}{p}+\frac{1}{m}-\frac{1}{q}\Bigr)<1.$$
Combining \eqref{XpXq}, \eqref{varconstu}, \eqref{condpqh} we obtain 
$$\begin{aligned}
\|v(t+\tau)\|_{X_q}
&\le C\tau^{-\gamma}\|v(t)\|_{X_p}+C\int_0^\tau (\tau-s)^{-\nu}\|h(s)\|_{X_r}\, ds\\
&\le C\tau^{-\gamma}\|v(t)\|_{X_p}+C\tau^{1-\nu}M\sup_{s\in(t,t+\tau)}\|v(s)\|^{1-\frac{1}{m}}_{X_p},
\end{aligned} $$
where $C>0$ is a generic constant depending only on $\Omega,p,q,m$.
Setting $N_p(a,b)=\sup_{s\in (a,b)}\|v(s)\|_{X_p}$ and taking supremum for $t\in(0,T-\tau)$, we get
\be{condpq2}
N_q(\tau,T)\le C\tau^{-\gamma}N_p(0,T)+CMN^{1-\frac{1}{m}}_p(0,T),\quad 0<\tau\le 1, \ \tau<T.
\ee
Now select $\eta\in\bigl(0,\min\bigl(1,\frac{1}{m_0}-\frac{1}{m}\bigr)\bigr)$ such that $1/\eta$ is noninteger,
put $N=1+[1/\eta]\ge 2$ and define $p_i=(1-i\eta)^{-1}$ for $i=0,\cdots,N-1$ and $p_N=\infty$.
For $i=1,\cdots,N$, since $p=p_{i-1}$ and $q=p_i$ satisfy \eqref{condpq}, we may apply \eqref{condpq2}
with $\tau$ replaced by $\tau/N$, which iteratively yields an estimate of 
$N_{p_i}(i\tau/N,T)$ in terms of $M$, hence \eqref{condpq3} for $i=N$.
 \end{proof}

We are now in a position to prove Theorem~\ref{thm-subsol2}.

\begin{proof}[Proof of Theorem~\ref{thm-subsol2}]
We split the proof into several steps for clarity.

\smallskip

{\bf Step 1.} {\it Preliminaries.} We first consider the case when $u_0, v_0\in L^\infty(\Omega)$ 
satisfy $u_0\ge v_0\ge 0$ and $u_0\ge v_0+\eps f(v_0)$.
We shall show that $v$ is bounded.
The idea is to use Lemma~\ref{lem-subsol} iteratively to show that arbitrary powers of $v$ are dominated by $u$,
hence are bounded in $L^1_\delta$ or $L^1_{ul}$ by Lemma~\ref{lem-UB},
and to conclude by means of Lemma~\ref{lem-smoothing}.
In order to carry out the iteration, we introduce the map 
$$\phi=Id+\eta f: [0,\infty)\to[0,\infty)$$
for suitable $\eta>0$ and will consider ``intermediate'' initial data
of the form $w_{i,0}=\phi^{(i)}(v_0)$ for $i\ge 1$,
where $\phi^{(i)}$ denotes the $i$-th iterate of $\phi$.
Namely, fixing an integer $k\ge 2$, we set 
\be{compuv0}
M_0=\|v_0\|_\infty,\quad 
L=\sup_{0\le s\le 2M_0} f'(s) \in [0,\infty), 
\quad \eta=\frac{1}{2k-1}\min(\eps,L^{-1}).
\ee
We also note that $f$ is nondecreasing in view of our assumptions.

\smallskip

{\bf Step 2.} {\it Comparison between $\phi^{(i)}$ and $f$.} 
We claim that
\be{compuv1}
\phi^{(i)}(s)\le s+(2i-1)\eta f(s),\quad 0\le s\le M_0,\quad i=1,\dots,k.
\ee
Assume that \eqref{compuv1} is true for a given $i\in \{1,\dots,k-1\}$ and let $s\in [0,M_0]$.
By \eqref{compuv0}, we have $f(s)\le Ls$, hence
$s+(2i-1)\eta f(s)\le (1+(2i-1)L\eta)M_0\le 2M_0$.
Since $\phi$ is nondecreasing, we deduce from the mean-value inequality that
$$f(\phi^{(i)}(s))\le f(s+(2i-1)\eta f(s)) 
\le f(s)+L\cdot (2i-1)\eta f(s)\le 2f(s),$$
hence
$$\phi^{{(i+1)}}(s)=\phi^{(i)}(s)+\eta f(\phi^{(i)}(s))\le s+(2i-1)\eta f(s)+2\eta f(s)=s+(2i+1)\eta f(s).$$
Since \eqref{compuv1} is true for $i=1$, the claim follows by induction.

\smallskip

{\bf Step 3.} {\it Comparison between $u$ and iterates of $v$.} 
We claim that
\be{compuv}
 \phi^{(k)}(v)\le u  \quad\hbox{in $\Omega\times(0,\infty)$}.
\ee
Denote by $w_i$ the solution of \eqref{MPf} with initial data 
$w_{i,0}=\phi^{(i)}(v_0)$ for $i=0,\dots,k$, hence $w_0=v$.
For $i\in \{0,\dots,k\}$, by  \eqref{compuv0}, \eqref{compuv1}, we have
$w_{i,0}=\phi^{(i)}(v_0)\le v_0+\eps f(v_0)\le u_0$.
It follows from the comparison principle that $w_i$ is global and
\be{compuv2}
w_i\le u \quad\hbox{in $\Omega\times(0,\infty)$},\quad i=0,\dots,k.
\ee
On the other hand, for $i\in \{0,\dots,k-1\}$, since $w_{i+1,0}=\phi(w_{i,0})$, we deduce from Lemma~\ref{lem-subsol} that
$w_{i+1}\ge \phi(w_i)$ in $\Omega\times(0,\infty)$, hence $w_k\ge \phi^{(k)}(v)$.
This combined with \eqref{compuv2} for $i=k$ gives \eqref{compuv}.

\smallskip

{\bf Step 4.} {\it Uniform boundedness of $v$.} 
We claim that, for each integer $k\ge 1$, there exists a constant $C_k>0$,
depending only on $k, f, \Omega, M_0, \eps$, such that
\be{itergi}
\frac{f^k(s)}{s^{k-1}}\le C_k\bigl(1+\phi^{(k)}(s)\bigr),\quad s>0.
\ee
Property \eqref{itergi} for $k=1$ is true with $C_1=1/\eta$. 
Assume it is true for some $k\ge 1$. 
Set $g(s)=\frac{f(s)}{s}$.
Since, by assumption \eqref{f10}, $g$ is nondecreasing and $\lim_{s\to\infty} g(s)=\infty$, 
it follows that, for all $s\ge1$, 
we have
$\phi^{(k)}(s)\ge s$ and $s g^k(s)\le 2C_k\phi^{(k)}(s)$.
Then, for all $s\ge1$, 
we obtain
$$\phi^{(k+1)}(s)\ge \eta f(\phi^{(k)}(s))=\eta \phi^{(k)}(s)g(\phi^{(k)}(s))
\ge \eta \phi^{(k)}(s)g(s) \ge \frac{\eta}{2C_k} s g^{k+1}(s),$$  
which guarantees \eqref{itergi} for $k+1$, hence for all $k\ge 1$ by induction.

Now, by \eqref{compuv} and \eqref{itergi}, it follows that, for all $k\ge 1$,
\be{compuv3}
\frac{f^k(v)}{v^{k-1}}\le C_k(u+1) \quad\hbox{in $\Omega\times(0,\infty)$}.
\ee
On the other hand, by Lemma~\ref{lem-UB},
the global solution $u$ is bounded in $X_1$, with a bound depending only on $f, \Omega$.
Choosing $k$ to be the smallest integer $k>(n+1)/2$, 
it follows from \eqref{compuv3} that 
the assumptions of Lemma~\ref{lem-smoothing} with $W(x,t):=|v|^{\frac{1}{k}-1}f(v)$ 
are satisfied for 
$m=k$ and $M$ depending only on $f, \Omega, M_0, \eps$.
Consequently, by Lemma~\ref{lem-smoothing}, for each $\tau>0$, there exists $M_2>0$, depending only on $f, \Omega, M_0, \eps, \tau$, such that
$$ \sup_{t\ge \tau}\|v(t)\|_{L^\infty(\Omega)} \le  M_2.$$
Since, by standard local theory, $\|v_0\|_\infty\le M_0$ implies $\|v(t)\|_\infty\le \|v_0\|_\infty+1$ for $t\in(0,t_0]$ with $t_0=t_0(f,\Omega,M_0)>0$,
this yields \eqref{AEepsf}.
  
\smallskip

{\bf Step 5.} {\it Conclusion.}  
We observe that what we just proved implies
the boundedness of $v$ in both assertions (i) and (ii),
as well as (iii).
Indeed, assume $\Omega$ bounded or $\Omega=\R^n$,
 and $u_0\ge \lambda v_0$ for some $\lambda>1$. Then, since $z_0:=f(v_0)/v_0\in L^\infty(\Omega)$ 
owing to $f(0)=0$ and $f\in C^1([0,\infty))$,
we have  $u_0\ge v_0+\eps f(v_0)$ with $\eps=(\lambda-1)(1+\|z_0\|_\infty)^{-1}$.
Moreover, if $\Omega$ is bounded and $u_0\ge v_0$, $u_0\not\equiv v_0$, then, since $f(s)/s$ is 
nondecreasing owing to $f(0)=0$ and $f$ convex, it follows from the proof of \cite[Lemma 17.9]{QS19} that
$u(\cdot,\tau)\ge \lambda v(\cdot,\tau)$ for some $\tau>0$ and some $\lambda>1$,
and we are reduced to the previous case after a time shift.

Finally assume $\Omega$ is bounded and $f$ strictly convex.
Set $w_0:=(u_0+v_0)/2$ and let $w$ denote the solution with initial data $w_0$.
It follows from assertion (iii) that $v$ and $w$ are global and bounded.
Moreover, as above, we have $w(\cdot,\tau)\ge \lambda v(\cdot,\tau)$ for some $\tau>0$ and some $\lambda>1$.
Since $\frac{f(s)}{s}$ is nondecreasing by assumption \eqref{f10}, it follows that $\lambda v$ is a subsolution, hence 
\be{compvw}
w\ge \lambda v\quad\hbox{in $\Omega\times(\tau,\infty)$}.
\ee
Assume for contradiction that the property $\lim_{t\to\infty}\|v(\cdot,t)\|_\infty=0$ fails.
Hence there exist $\delta>0$ and a sequence $t_j\to\infty$ such that
\be{compvw2}
\|v(\cdot,t_j)\|_\infty\ge \delta.
\ee
By standard dynamical systems arguments (see, e.g.,~\cite[Proposition~53.6 and Example~53.7]{QS19}),
there exist nonnegative steady states $V, W$ and a subsequence $t'_j$ such that, as $j\to\infty$, $v(t_j)$ and $w(t_j)$ 
converge uniformly to $V$ and $W$, respectively.
We then have $V\not\equiv 0$ by \eqref{compvw2}, hence $V>0$ in $\Omega$, and $W\ge\lambda V$ by \eqref{compvw}.
But, since $f$ is strictly convex, this contradicts the nonexistence of multiple 
ordered positive steady states (cf.~\cite[Proposition~19.8]{QS19}).
The proof is complete. 
\end{proof}

\begin{proof}[Proof of Corollary~\ref{cor-nonrad}(i)-(ii)]
In the case $\Omega$ bounded,
the assertions 
follow from Theorem~\ref{thm-subsol2}. 

Thus consider the case $\Omega=\R^n$. If $u_0$ is radial, radially nonincreasing
and $f(u_0)\in L^1$, then Theorem~\ref{thm-subsol2} implies that
the solution $w$ with initial data $w_0=u_0+\eps f(u_0)$
blows up in finite time for any $\eps>0$,
hence $u_0$ is a threshold due to
the proof of \cite[Theorem~2]{Q17}
(which remains true for convex functions $f$).
Consequently, assertion (i) is true.

Now take $v_0\le u_0$, $v_0\not\equiv u_0$,
and assume to the contrary that $v$ is unbounded.
Then $u(\cdot,1)$ is radially decreasing and $v(\cdot,1)<u(\cdot,1)$,
hence we may find $w_1$ radial and radially nonincreasing such that 
$v(\cdot,1)\le w_1\le u(\cdot,1)$, $w_1\not\equiv u(\cdot,1)$.
Since $w$ is global unbounded, assertion (i) implies that $u$ blows up in finite time,
a contradiction.
\end{proof}

The proof of Corollary~\ref{cor-nonrad}(iii) is postponed to Section~\ref{sec-3} after the proof of Theorem~\ref{thm-nonrad-3},
since it uses some notation and arguments from the latter.

\begin{proof}[Proof of Theorems~\ref{thm-introA} and~\ref{thm-intro}]
(i) Let $u^*$ be an $\eps$-threshold with initial data $u_0^*$
and $u$ be an $\eps f$-subthreshold with initial data $u_0$.
Then there exists $\eps>0$ such that $u_0+\eps f(u_0)\le u_0^*$.
Choose $\lambda_j\in(0,1)$ such that $\lambda_j\nearrow 1$
and let $u_j^*$ and $u_j$ be the (global) solutions with initial data
$\lambda_j u_0^*$ and $\lambda_j u_0$.
Since $f(s)/s$ is nondecreasing, we have
$$ \lambda_j u_0+\eps f(\lambda_j u_0)\le \lambda_j(u_0+\eps f(u_0))\le\lambda_j u_0^*.$$
By Theorem~\ref{thm-subsol2}(iii), $u_j$ satisfies an $L^\infty$-bound, which depends on $\eps$ but not on $j$.
Since $u_j\nearrow u$, the solution $u$ is bounded.

Next assume that $\Omega$ is bounded.
Then $u^*$ is a threshold.
Let $u$ be a subthreshold with initial data $u_0$.
Fix $\delta>0$ small and consider $\tilde u_0^*:=u^*(\cdot,\delta)$ and $\tilde u_0:=u(\cdot,\delta)$
instead of $u_0^*$ and $u_0$, respectively.
The solution $\tilde u^*(t):=u^*(t+\delta)$ is a threshold: 
In fact, if, for example, $v_0\ge \tilde u_0^*$ and $v_0\not\equiv \tilde u_0^*$,
then the corresponding solution $v$ satisfies $v(\cdot,\delta)\ge(1+\eps)\tilde u^*(\cdot,\delta)=(1+\eps)u^*(\cdot,2\delta)$ for some $\eps$,
and we can find $w_0\ge u_0^*$, $w_0\not\equiv u_0^*$, such that the corresponding solution $w$
satisfies $w(\cdot,2\delta)\le(1+\eps)u^*(\cdot,2\delta)$.
Since $w$ blows up in finite time and $v(\cdot,t)\ge w(\cdot,t+\delta)$ for $t\ge\delta$, 
$v$ has to blow up in finite time too.
We can also find $\eps'>0$ such that $(1+\eps')\tilde u_0\le\tilde u_0^*$,
so that the first part of the proof shows that $\tilde u$ (hence also $u$) is bounded.

Now assume that $\Omega$ is bounded and $f$ is strictly convex.
The solution $w$ with initial data $w_0=(u_0^*+u_0)/2$ is a subthreshold, hence it is global. 
The decay of $u$ then follows from Theorem~\ref{thm-subsol2}(i).

\smallskip

(ii) This follows from Corollary~\ref{cor-nonrad} and Theorem~\ref{thm-subsol2}(ii).

\smallskip

(iii) Assume that  
$\Omega=\R^n$, $u^*$ is an $\eps f$-threshold with initial data $u_0^*$, \eqref{assL1} is true,
and $u$ is a subthreshold with initial data $u_0$. 
Then the arguments in the proof of \cite[Theorem~4]{Q17} guarantee that $u^*$ is a threshold.
In fact, we will use similar arguments to show that $u$ is bounded.
Choose $\delta>0$ small. Then $u(\cdot,\delta)<u^*(\cdot,\delta)$ and $u^*(\cdot,\delta)$ is
radially decreasing, hence there exists $\nu>0$ small such that
$w_0(x):=\min(u^*(x,\delta),u^*(0,\delta)-\eta)$ satisfies $u(\cdot,\delta)\le w_0$.
Set $\eta:=\|u^*(\cdot,\delta)-w_0\|_1>0$.
If $\eps$ is small enough and $v_0=u_0-\eps f(u_0)$, then the corresponding $v$ is a radially
nonincreasing $\eps f$-subsolution
(so that $v$ stays bounded by the first part of the proof) and $\|u^*(\cdot,\delta)-v(\cdot,\delta)\|_1<\eta$.
Consequently, $\int_{\R^n}(v(\cdot,\delta)-w_0)\,dx>0$. In addition we have $z(v(\cdot,\delta)-w_0)=1$,
where $z(\cdot)$ denotes the zero number in $(0,\infty)$.
Let $w$ be the solution with initial data $w_0$.
The same arguments as in the proof of \cite[Proposition~2]{Q17} guarantee 
that $\int(v(\cdot,t+\delta)-w(\cdot,t))\,dx$ stays positive and $w(0,t)\le v(0,t+\delta)$.
This implies that $w$ (hence also $u$) is bounded.
\end{proof}

\section{Nonconvex nonlinearities with critical growth}
\label{sec-3}

In this section we will assume that $p=p_S$
and the nonlinearity $f:[0,\infty)\to[0,\infty)$ is $C^1$ and satisfies 
\begin{equation} \label{f1}
(0,\infty)\to\R:s\mapsto \frac{f(s)}s \ \hbox{ is nondecreasing},
\end{equation}
\begin{equation} \label{f2}
\lim_{\lambda\to\infty}\frac{f(\lambda s)}{f(\lambda)}=s^p,\quad\hbox{locally uniformly for }s\in(0,\infty),
\end{equation}
there exist $\eta>0$ and $C_\eta\ge0$ ($C_\eta=0$ if $\Omega=\R^n$) such that
\begin{equation} \label{f3}
f(s)s\ge(2+\eta){\mathcal F}(s)-C_\eta,\quad\hbox{where }\ {\mathcal F}(s):=\int_0^s f(z)\,dz,
\end{equation}
and there exist 
$C_f>0$ and $\tilde C_f\ge0$ ($\tilde C_f=0$ if $\Omega=\R^n$) such that
\begin{equation} \label{f4}
f(s)\le C_fs^p+\tilde C_f.
\end{equation}
(We could take $\tilde C_f>0$ if $\Omega=\R^n$ and $u_0\in H^1$.)
We will consider classical solutions of \eqref{MPf}
with initial data 
\begin{equation} \label{ass-u0}
u_0\in C(\bar\Omega)\cap L^\infty(\Omega)\cap{\mathcal E}(\Omega), \quad u_0\ge0, 
\end{equation}
where   
$$ {\mathcal E(\Omega)}:=\{w\in L^{p+1}(\Omega): \nabla w\in L^2(\Omega)\}.$$
It is known (see the arguments in \cite[Example~51.28]{QS19}) that the energy function
$$ t\mapsto E(u(\cdot,t)):= \frac12\int_\Omega|\nabla u(x,t)|^2\,dx 
  -\int_\Omega{\mathcal F}(u(x,t))\,dx $$
is well defined, nonincreasing and
$$ E(u(\cdot,t_1))-E(u(\cdot,t_2))=\int_{t_1}^{t_2}\int_\Omega u_t^2\,dx\,dt, \qquad t_1<t_2. $$
In addition, if $E(u(\cdot,t))<0$ for some $t\ge0$, then $T(u_0)<\infty$.
Consequently, the energy function is nonnegative if the solution is global.

\begin{theorem} \label{thm-nonrad-3}
Assume \eqref{ass-Omegap0}, $p=p_S$ and \eqref{f1}--\eqref{ass-u0}.
Let $u$ be global and unbounded. Then $u$ is an $\eps$-threshold.
More precisely, let $v$ be the solution with initial data $v_0=\lambda u_0$,
where $\lambda>0$. Then we have the following:

\begingroup
    \addtolength{\leftmargini}{-10pt}
\begin{itemize}
\item[(i)] If $\lambda>1$, then $v$ blows up in finite time.

\smallskip
\item[(ii)] If $\lambda<1$, then $v$ is global and bounded.
In addition, if $f(s)=s^{p_S}$, 
then \hfill\break
\hbox{$\lim_{t\to\infty}\|v(\cdot,t)\|_\infty=0$.} 
\end{itemize}
\endgroup
\end{theorem}

In the proof of Theorem~\ref{thm-nonrad-3} we will need several lemmas. 
The first one deals with intersection properties of positive solutions
of the equation 
\begin{equation} \label{eq-GS-3}
\Delta U+U^p=0\quad\hbox{in }\ \R^n,
\end{equation}
where $p=p_S$.
Recall from \cite[Section~9]{QS19} that 
$U_1(y):=\Bigl(\frac{n(n-2)}{n(n-2)+|y|^2}\Bigr)^{(n-2)/2}$ is the unique positive radial solution of
\eqref{eq-GS-3}  
satisfying $U_1(0)=1$.
Let $U_A(y):=AU_1(A^{(p-1)/2}y)$, $A>0$, be the rescaled solution satisfying
$U_A(0)=A$. Notice that if $A\ne1$, then
\begin{equation} \label{intersectU1A-3}
U_1(y)=U_A(y) \quad\hbox{if and only if}\quad |y|=A_1:=A^{-1/(n-2)}\sqrt{n(n-2)},
\end{equation}
and if $A<1$, then 
\begin{equation} \label{U1Apos-3}
U_1(y)>U_A(y) \quad\hbox{if and only if}\quad |y|<A_1.
\end{equation}
In addition, if $c\in(0,1)$, then
\begin{equation} \label{U1small-3}
U_1(y)\le c \quad\hbox{if and only if}\quad |y|\ge R_c:=\sqrt{n(n-2)(c^{-2/(n-2)}-1)}.
\end{equation}
We also note that:
\begin{equation} \label{asymptUA2}  
\hbox{If $B>A>0$ and $y_0\in\R^n$, then $U_A$ and $U_B(y-y_0)$ intersect.}
\ee
Indeed, since
$$U_A(y)=A\big(1+c(n)A^{\frac{4}{n-2}}|y|^2\big)^{\frac{2-n}{2}}\sim {\tilde c}(n)A^{-1}|y|^{2-n},\ \hbox{ as } |y|\to\infty,$$
we have $W(y):=U_B(y-y_0)-U_A(y)\sim {\tilde c}(n)(B^{-1}-A^{-1})|y|^{2-n}<0$ as $|y|\to\infty$
and $W(y_0)\ge B-A>0$.
The following lemma gives a more quantitative version of property \eqref{asymptUA2} that is needed for
the proof of Theorem~\ref{thm-nonrad-3}.

\begin{lemma} \label{lem-intersect-3}
Let
\begin{equation} \label{ass-lem-intersect-3}
A\in\Bigl[\frac12,1\Bigr),\quad R>R^*:=2^{(n-1)/(n-2)}\sqrt{n(n-2)},\quad |y_0|<\frac R2.
\end{equation}
Then $R/2>\max(R_{1/2},A_1)$ and
there exists $y$ such that
\begin{equation} \label{U1Ay0-3}
|y|<\frac R2 \quad\hbox{and}\quad U_1(y)=U_A(y-y_0).
\end{equation}
\end{lemma}

\begin{proof}
The inequality $R/2>\max(R_{1/2},A_1)$ is obvious.

If $y_0=0$, then the remaining assertion follows from \eqref{intersectU1A-3}
since $R/2>A_1$.
Hence we may assume $y_0\ne0$.
Set
$$ \varphi(t):=U_1(ty_0)-U_A((t-1)y_0), \quad y\geq0,$$
and notice that $\varphi(0)\ge1-A>0$.

If $\varphi(1)\le0$, then there exists $t_0\in(0,1]$ such that
$\varphi(t_0)=0$ and $|t_0y_0|\leq|y_0|<R/2$,
hence it is sufficient to choose $y:=t_0y_0$.

Next assume $\varphi(1)>0$, i.e.~$U_1(y_0)>A$.
Then \eqref{U1small-3} implies 
$$|y_0|<\sqrt{n(n-2)(A^{-2/(n-2)}-1)}<A_1,$$ 
hence $t_1:=\frac{A_1}{|y_0|}>1$.
Now \eqref{intersectU1A-3} and the fact that $U_A$ is radially decreasing guarantee
$$ U_1(t_1y_0)=U_A(t_1y_0)<U_A((t_1-1)y_0), $$
so that $\varphi(t_1)<0$. Consequently, there exists $t_0\in(1,t_1)$ such that
$\varphi(t_0)=0$. In addition,
$|t_0y_0|<|t_1y_0|=A_1<R/2$,
hence it is sufficient to choose $y:=t_0y_0$.
\end{proof}

Property \eqref{U1small-3} implies that
$U_1(y)>\frac12$ if and only if $|y|<R_{1/2}$. 
Set
\begin{equation} \label{c0-3}
c_0=c_0(n):=\frac1{2C_f^{(n-2)/2}}\int_{|y|<R_{1/2}}\Bigl(U_1^{p+1}(y)-\Bigl(\frac12\Bigr)^{p+1}\Bigr)\,dy > 0. 
\end{equation}

\begin{lemma} \label{lem1-3}
Let $u,u'$ be global unbounded solutions with initial data $u_0$ and $u'_0$
satisfying \eqref{ass-u0}. Let $\eps>0$ and $u_0\ge(1+\eps)u'_0$.
Let $c_1,C_1>0$ and assume  that, for some sequence $t_k\to\infty$,
we have $M_k:=\|u(\cdot,t_k)\|_\infty\to\infty$, 
$\delta_k\ge c_1M_k/f(M_k)$, and 
$$ \|u(\cdot,t)\|_\infty\le C_1M_k\quad\hbox{for}\quad t\in[t_k-\delta_k,t_k].$$
Then 
\begin{equation} \label{vuc0-3}
\int_\Omega (f(u)u)(x,t_k)\,dx\ge c_0+\int_\Omega (f(u')u')(x,t_k)\,dx
\end{equation}
for $k$ large enough,
where $c_0$ is defined in \eqref{c0-3}.
\end{lemma}

\begin{proof}
Assume to the contrary that \eqref{vuc0-3}
fails for some subsequence of $\{t_k\}$.
Passing to this subsequence we may assume that
\begin{equation} \label{vuc0-3contrary}
 \int_\Omega (f(u)u)-f(u')u')(x,t_k)\,dx<c_0,\quad\hbox{for all $k$.}
\end{equation}
 
Since $(1+\eps)u'$ is a subsolution owing to \eqref{f1}, we have $u\ge(1+\eps)u'$.
Choose $x_k\in\Omega$ such that $u(x_k,t_k)-M_k\to0$.
Denote $\nu_k:=\sqrt{M_k/f(M_k)}$ and set also
\begin{equation} \label{wk-3}
\begin{aligned}
  w_k(y,s) &:= M_k^{-1} u(x_k+\nu_ky,t_k+\nu_k^2s), \\
  w'_k(y,s) &:= M_k^{-1} u'(x_k+\nu_ky,t_k+\nu_k^2s).
\end{aligned} 
\end{equation}
Then $w_k,w'_k$ are solutions of the equation
\begin{equation} \label{eq-ub10-3}
 w_s -\Delta w = \frac{1}{f(M_k)}f(M_kw) \quad \hbox{in}\quad\Omega_k\times[-c_1,0],
\end{equation}
where $\Omega_k=\nu_k^{-1}(\Omega-x_k)$.
Moreover, 
\begin{equation} \label{wk-bound}
0\le(1+\eps)w'_k\le w_k\le C_1,\quad  w_k(\cdot,0)\leq 1\quad\hbox{and}\quad w_k(0,0)\to1.
\end{equation}
In particular, \eqref{f1} and \eqref{f2} guarantee that the right-hand side in \eqref{eq-ub10-3}
is bounded and $\frac{1}{f(M_k)}f(M_kz)\to z^p$ as $k\to\infty$, uniformly for $z\in[0,C_1]$.  
Since $u,u'$ are global, the corresponding energies are nonincreasing and nonnegative,
hence 
\begin{equation} \label{epsk}
\eps_k:=\int_{t_k-\delta_k}^{t_k}\int_\Omega (u^2_t+(u')_t^2)\,dx\,dt\to0 \quad\hbox{as }\ k\to\infty.
\end{equation}
Since 
$$ \int_{t_k-c_1\nu_k^2}^{t_k}\int_\Omega u^2_t\,dx\,dt
 = \int_{-c_1}^{0}\int_{\Omega_k}(w_k)^2_s\,dy\,ds 
$$ due to $p=p_S$
(and similarly for $u'$),
\eqref{epsk} implies
$$ \int_{-c_1}^{0}\int_{\Omega_k}((w_k)^2_s+(w'_k)_s^2)\,dy\,ds\le\eps_k\to 0 \quad\hbox{as }\ k\to\infty.$$
In the same way as in \cite{G86} we obtain a subsequence of $w_k$ converging locally uniformly 
to the solution $w$ of $-\Delta U=U^p$ in $\R^n$ or to the solution 
$\tilde w$ of this equation in a half-space
satisfying the homogeneous Dirichlet boundary conditions.
Since the existence of $\tilde w$ contradicts 
\cite[Theorem~8.2]{QS19}, we deduce from
\cite[Theorems~8.1(ii) and~9.1]{QS19} and \eqref{wk-bound} that 
\begin{equation} \label{wU1-3}
\begin{aligned}
&\hbox{$w_k(\cdot,0)\to w=U_1$ locally uniformly as $k\to\infty$,} \\ 
&\hbox{$B_R^k:=\{y\in\Omega_k:|y|<R\}=B_R:=\{y\in\R^n:|y|\le R\}$ for $k$ large.}
\end{aligned}
\end{equation}
The same arguments show that, passing to a subsequence, we have
\begin{equation} \label{wU1-3b}
\begin{aligned}
&\hbox{$w'_k(\cdot,0)\to w'$ locally uniformly as $k\to\infty$, where either $w'=0$ or} \\ 
&\hbox{$w'(y)=U_B(y-y^*)$ for some $y^*\in\R^n$ and $B\in(0,1/(1+\eps)]$.}
\end{aligned}
\end{equation}

Fix $R>R^*$ such that $U_1(y)<1/6$ for $|y|=R/2$ (see \eqref{U1small-3}).
Notice that if $k$ is large enough, then \eqref{wU1-3} implies that 
$N_k:=\{x:|x-x_k|\le R\nu_k\}\subset\Omega$ and
$|w_k(y,0)-U_1(y)|<1/12$ for $|y|\le R$, hence
\begin{equation} \label{uest-3}
w_k(y,0)<1/4\quad\hbox{for}\quad R/2\le|y|\le R.
\end{equation}
We have $1/(1+\eps)\ge \max_{|y|\le R}w'_k(y,0)=w'_k(y_k,0)$ for some $y_k$,
$|y_k|\le R$.
 We claim that 
\begin{equation} \label{uest-3b}
w'_k(y,0)\le 1/2,\quad |y|\le R.
\ee

Assume for contradiction that $w'_k(y_k,0)>1/2$ for some subsequence. Then \eqref{uest-3} implies 
$|y_k|<R/2$.
Passing to a subsequence we may assume 
\begin{equation} \label{yky0-3}
y_k\to y_0, \quad\hbox{where}\quad |y_0|\le\frac R2.
\end{equation}
Since $w'_k(y_k,0)\in(\frac12,\frac1{1+\eps}]$,
passing to a subsequence we may also assume
$$w'_k(y_k,0)\to A,\quad\hbox{where}\quad A\in\Bigl[\frac12,\frac1{1+\eps}\Bigr].$$
We have
\begin{equation} \label{z1-3}
w'_k(\cdot,0)\le\frac1{1+\eps},\quad
w'_k(y_k,0)=\max_{|y|\le R}w'_k(y,0)\to A,
\end{equation}
hence \eqref{yky0-3} implies 
\begin{equation} \label{z2-3}
w'(y_0)=A=\max_{|y|\le R} w'.
\end{equation}
By \eqref{wU1-3b}, $y^*$ is the unique local maximum of $w'$ and, since $|y_0|\le R/2$, we have
$y_0=y^*$, $A=B$ and
$w'_k(y,0)\to w'(y)=U_A(y-y_0)$.
Notice that $w_k\ge(1+\eps)w'_k$,
$(w_k-w'_k)(y,0)$ converges uniformly for $|y|\le R$
to $U_1(y)-U_A(y-y_0)$,
but $U_1(y)-U_A(y-y_0)=0$ for some $|y|<R/2$ due to Lemma~\ref{lem-intersect-3},
which yields a contradiction. 
 This proves~\eqref{uest-3b}.
  
 Now take any $\zeta>0$. 
For $k$ large enough, using $w_k\ge w'_k$, \eqref{f1}, \eqref{f2}, \eqref{f4}, \eqref{wU1-3} and \eqref{uest-3b}, 
we obtain
$$ 
\begin{aligned}
\int_{N_k}&(f(u) u-f(u') u')(x,t_k)\,dx \\ 
&=\nu_k^n\int_{|y|\le R}(f(M_kw_k)M_kw_k-f(M_kw'_k)M_kw'_k)(y,0)\,dy \\
&\ge\nu_k^n\int_{|y|<R_{1/2}}(f(M_kw_k)M_kw_k-f(M_kw'_k)M_kw'_k)(y,0)\,dy \\
&=M_k^{(n+2)/2}f(M_k)^{(2-n)/2}\int_{|y|<R_{1/2}}\Bigl(\frac{f(M_kw_k)}{f(M_k)}w_k-\frac{f(M_kw'_k)}{f(M_k)}w'_k\Bigr)(y,0)\,dy \\
&\ge \frac1{((1+\zeta)C_f)^{(n-2)/2}}\int_{|y|<R_{1/2}}\Bigl(U_1^{p+1}(y)-\Bigl(\frac12\Bigr)^{p+1}-\zeta\Bigr)\,dy,
\end{aligned}
$$
which contradicts \eqref{vuc0-3contrary} for  $\zeta$ small enough.
\end{proof}

\begin{lemma} \label{lemM2}
Let $f$ satisfy \eqref{f1} and \eqref{f2}.
Assume $y'\le f(y)$ on $[t-\delta,t]$, $y(t-\delta)\le M/2$ and $y(t)=M$.
If $M>0$ is large enough, then $\delta\ge\frac{M}{2f(M)}$.
\end{lemma}

\begin{proof}
Assume to the contrary $\delta<\frac{M}{2f(M)}$ and fix $\zeta\in(0,1)$.
Set $\omega(\tau):=y(\tau)/M$ and $\omega^+(\tau):=\max\{\omega(\tau),\frac12\}$. 
Assuming $\omega(\tau)\le1$ for $\tau\in[t-\delta,s]$, $s\le t$, 
and $M$ large, we obtain
$$ 
\begin{aligned}
y(s)-y(t-\delta) &\le \int_{t-\delta}^sf(\omega(\tau) M)\,d\tau
   \le  \int_{t-\delta}^sf(\omega^+(\tau) M)\,d\tau
   \le f(M)\int_{t-\delta}^s((\omega^+(\tau)^p+\zeta)\,d\tau \\
  &\le f(M)\delta(1+\zeta)<M, 
\end{aligned} $$
which contradicts $y(t)=M$.
\end{proof}

\begin{lemma} \label{lem-3}
Let $v^j$ be global unbounded solutions with initial data $v^j_0$ 
satisfying \eqref{ass-u0}, $j=0,1,\dots,m$, and 
$v^{j+1}_0>v^j_0$ for $j=0,\dots,m-1$.
Then there exist $K>2$, $c>0$, $\delta_k^j>0$ and $t_k\to\infty$ with the following properties:
\begin{equation} \label{assert-lem-3}
\begin{aligned}
M_k^j &:=\|v^j(\cdot,t_k)\|_\infty\to\infty, \quad \delta_k^j\ge c\frac{M_k^j}{f(M_k^j)}, \\
 \|v^j(\cdot,t)\|_\infty &\le K^2M_k^j\quad\hbox{for}\quad t\in[t_k-\delta_k^j,t_k].
\end{aligned}
\end{equation}
\end{lemma}

\begin{proof}
First notice that \cite[Remark~7.1(i)]{QS25} implies that
\begin{equation} \label{fMM}
 \frac{f(\lambda s)}{f(\lambda)}\geq c_fs^{(p+1)/2} \quad\hbox{for $\lambda,s\ge1$}.
\end{equation}
Choose $K>2$ such that $K^{(p-1)/2}>(1+2/c_f)$.
We will first prove the existence of $t_{k,j}$ with the following properties:
Set 
$$M_{k,j}:=\|v^j(\cdot,t_{k,j})\|_\infty,\quad \delta_j:=\frac1{3^j4K^{p-1}}, 
\quad \delta_{k,j}:=\delta_j \frac{M_{k,j}/K}{f(M_{k,j}/K)},$$
$$I_{k,j}:=[t_{k,j}-\delta_{k,j},t_{k,j}],\quad J_{k,j}:=[t_{k,j}-\frac34\delta_{k,j},t_{k,j}].$$
Then $t_{k,j}\to\infty$ and $M_{k,j}\to\infty$ as $k\to\infty$, $J_{k,j}\subset J_{k,j-1}$ and
\begin{equation} \label{vjK}
 \frac1KM_{k,j}\le\|v^j(\cdot,t)\|_\infty\leq KM_{k,j}\quad\hbox{for}\quad t\in I_{k,j}.
\end{equation}
Notice that given $\tau\in J_{k,j}$, \eqref{vjK} implies
\begin{equation} \label{tauJ}
 \|v^j(\cdot,t)\|_\infty\le K^2\|v^j(\cdot,\tau)\|_\infty \quad\hbox{for}\quad t\in \Bigl[\tau-\frac14\delta_{k,j},\tau\Bigr].
\end{equation}
Since $\delta_{k,j}\ge\delta_j\frac{\|v^j(\cdot,\tau)\|_\infty}{f(\|v^j(\cdot,\tau)\|_\infty)}$ due to \eqref{vjK} and \eqref{f1}, 
the assertion of  the lemma will follow 
by choosing $t_k\in\bigcap_j J_{k,j}$ and $c:=\frac14\delta_m$. 

We will construct $t_{k,j}$ by induction w.r.t. $j$.
If $j=0$, then we choose $t_{k,0}\to\infty$ such that 
$M_{k,0}:=\|v^0(\cdot,t_{k,0})\|_\infty=\max_{t\le t_{k,0}}\|v^0(\cdot,t)\|_\infty$. Then $M_{k,0}\to\infty$
and the upper estimate in \eqref{vjK} with $j=0$ is trivially satisfied.
The lower estimate follows from Lemma~\ref{lemM2} and \eqref{f2} for $k$ large enough.

Next assume that $t_{k,0},\dots,t_{k,j}$ satisfy the assertion and we will find $t_{k,j+1}$.
Set $t^1:=t_{k,j}$ and $M_1:=\|v^{j+1}(\cdot,t^1)\|_\infty\geq M_{k,j}$.
If the upper estimate in \eqref{vjK} with $j$ replaced by $j+1$, $t_{k,j+1}=t^1$ and $M_{k,j+1}=M_1$ fails,
then we can find $t^2\in[t^1-\delta_{j+1}(M_1/K)/f(M_1/K),t^1]$ such that
$M_2:=\|v^{j+1}(\cdot,t^2)\|_\infty>KM_1$.
If the upper estimate in \eqref{vjK} with $j$ replaced by $j+1$, $t_{k,j+1}=t^2$ and $M_{k,j+1}=M_2$ fails,
then we can find $t^3\in[t^2-\delta_{j+1}(M_2/K)/f(M_2/K),t^2]$ such that
$\|v^{j+1}(\cdot,t^3)\|_\infty>KM_2>K^2M_1$.
This process has to stop after finitely many steps,
since otherwise $\|v^{j+1}(\cdot,t^i)\|_\infty$ would become unbounded as $i\to\infty$.
But this is impossible in view of
\begin{equation} \label{ti}
\begin{aligned}
t^{i+1} &\ge t^1-\delta_{j+1}\Bigl(\frac{M_1/K}{f(M_1/K)}+\frac{M_2/K}{f(M_2/K)}+\dots+\frac{M_i/K}{f(M_i/K)}\Bigr) \\
 &\ge t^1-\delta_{j+1}\Bigl(\frac{M_1/K}{f(M_1/K)}+\frac{KM_1/K}{f(KM_1/K)}+\dots+\frac{K^{i-1}M_1/K}{f(K^{i-1}M_1/K)}\Bigr) \\
 &\ge t^1-\delta_{j+1}\frac{M_1/K}{f(M_1/K)}\Bigl(1+K\frac{f(M_1/K)}{f(KM_1/K)}+\dots+K^{i-1}\frac{f(M_1/K)}{f(K^{i-1}M_1/K)}\Bigr) \\
 &\ge t^1-\delta_{j+1}\frac{M_1/K}{f(M_1/K)}\Bigl(1+\frac1{c_fK^{(p-1)/2}}+\frac1{c_f(K^2)^{(p-1)/2}}+\dots\Bigr) \\
 &= t^1-\frac13\delta_{j}\frac{M_1/K}{f(M_1/K)}\Bigl(1+\frac1{c_f(K^{(p-1)/2}-1)}\Bigr)\ge t^1-\frac{\delta_j}2\frac{M_1/K}{f(M_1/K)}
\end{aligned}
\end{equation}
due to \eqref{fMM} and our choice of $K$ and $\delta_j$.
Choose $t_{k,j+1}=t^i$ such that the upper estimate in \eqref{vjK} with $j$ replaced by $j+1$
is true. The corresponding lower estimate follows from Lemma~\ref{lemM2}.
Notice also that $J_{k,j+1}\subset J_{k,j}$, since
$t_{k,j+1}\le t_{k,j}$, $\delta_{k,j}\ge3\delta_{k,j+1}$ and \eqref{ti} implies
$t_{k,j+1}=t^i\ge t_{k,j}-\frac{\delta_{k,j}}2$.
\end{proof}

\begin{proof}[Proof of Theorem~\ref{thm-nonrad-3}]
(i)
Assume to the contrary that $v$ is a global solution. 

Set $w_0:=(u_0+v_0)/2$ 
and let $w$ denote the solution with initial data $w_0$.
Then $v_0\ge(1+\eps)w_0$ for some $\eps>0$ 
and, since $(1+\eps)w$ is a subsolution owing to \eqref{f1}, we have
$v\ge(1+\eps)w$.
We can assume $\eps<\eta$, where $\eta$ is the constant in \eqref{f3}.
Choose a positive integer $m$ such that 
\begin{equation} \label{choicem}
 \frac{\eps(\eta-\eps)}{2+\eta}c_0m>(1+\eps)^2E(w_0)+\frac{\eps C_\eta|\Omega|}{2+\eta},
\end{equation}
where $|\Omega|$ denotes the measure of $\Omega$ and $C_\eta|\Omega|=0$ if $\Omega=\R^n$.
Set $v^0_0=u_0$, $v^m_0=w_0$ and choose $\delta>0$ and
initial data $v^1_0\le v^2_0\le\dots\le v^{m-1}_0$ such that $v^{j+1}_0\ge(1+\delta)v^j_0$ 
for $j=0,\dots,m-1$. 
Let $v^j$ denote the solution with initial data $v^j_0$
(hence $v^0=u$, $v^m=w$).
Then $v^{j+1}\ge(1+\delta)v^j$ for $j=0,\dots,m-1$
and Lemma~\ref{lem-3} guarantees the existence of $t_k\to\infty$ such
that \eqref{assert-lem-3} is true.

Now Lemma~\ref{lem1-3} with $(u,u')$ replaced by $(v^{j+1},v^j)$, $j=0,1,\dots,m-1,$ 
implies
\begin{equation} \label{fwwc0}
\int_\Omega (f(w)w)(x,t_k)\,dx\ge c_0m\quad\hbox{for $k$ large},
\end{equation}
where we also used $w=v^m$, $u=v^0$.
Since ${\mathcal F}((1+\eps)w)\ge {\mathcal F}(w)+\eps f(w)w$
by the convexity of ${\mathcal F}$, we deduce from \eqref{f3} that
$$ {\mathcal F}((1+\eps)w)-(1+\eps)^2{\mathcal F}(w)
\ge \eps(f(w)w-(2+\eps){\mathcal F}(w))
\ge \eps\Bigl(\frac{\eta-\eps}{2+\eta}f(w)w-\frac{C_\eta}{2+\eta}\Bigr).$$ 
 By using \eqref{choicem} and \eqref{fwwc0} we obtain
$$\begin{aligned}
E((1+\eps)w(\cdot,t_k))&=(1+\eps)^2 E(w(\cdot,t_k)) \\
&\kern3mm +(1+\eps)^2\int_\Omega {\mathcal F}(w(x,t_k))\,dx 
-\int_\Omega{\mathcal F}((1+\eps)w(x,t_k))\,dx \\
&\le (1+\eps)^2 E(w_0)-\eps\int_\Omega\Bigl(\frac{\eta-\eps}{2+\eta}(f(w)w)(x,t_k)-\frac{C_\eta}{2+\eta}\Bigr)\,dx
<0.
\end{aligned}
$$
Now $v\geq(1+\eps)w$, \cite[Theorem 17.6]{QS19} and the comparison principle imply that $v$ blows up in finite time
- a contradiction.

(ii) Assertion (i) implies that $v$ is global and bounded.
If $f(s)=s^p$ and $\Omega$ is bounded, 
then $\lim_{t\to\infty}\|v(\cdot,t)\|_\infty=0$ by 
Theorem~\ref{thm-subsol2}(i). 

Finally assume that $f(s)=s^p$, $\Omega=\R^n$ and
$(1+2\eps)v_0=u_0$ 
for some $\eps\in(0,1/2)$. Set
$w_0:=(u_0+v_0)/2=(1+\eps)v_0$, hence $(1+\eps)v\le w\leq C_w$. 
Assume to the contrary that there exist $A>0$ and $t_k\to\infty$
such that $\|v(\cdot,t_k)\|_\infty\to A$.
Then similarly as above there exist $\delta>0$ and
$c_2>0$
such that $\|v(\cdot,t)\|_\infty\ge c_2A$ for $t\in[t_k-\delta,t_k]$.
Energy arguments show that 
$$\int_{t_k-\delta}^{t_k}\int_{\R^n}(v_t^2+w_t^2)\,dx\,dt \to 0 \quad\hbox{as }\ k\to\infty.$$
Choose $x_k$ such that $v(x_k,t_k)\ge(1-\delta_k)\|v(\cdot,t_k)\|_\infty$ for
some $\delta_k\to0$. Set $v_k(y,s):=v(x_k+y,t_k+s)$,
$w_k(y,s):=w(x_k+y,t_k+s)$ and notice that
$(1-\delta_k)\|v(\cdot,t_k)\|_\infty\le v_k(0,0)$, $v_k(\cdot,0)\le\|v(\cdot,t_k)\|_\infty\to A$,
$w_k(0,0)\ge(1+\eps)v_k(0,0)$ and $w_k(\cdot,0)\leq C_w$.
Passing to a subsequence 
as in the proof of Lemma~\ref{lem1-3} (cf.~\eqref{wU1-3}, \eqref{wU1-3b})
we find that
$v_k(y,0)\to U_A(y)$ and $w_k(y,0)\to U_B(y-y_0)$ locally
uniformly for some $y_0\in\R^n$, where $B>A$. 
Consequently, $U_B(y-y_0)>U_A(y)$.
This yields a contradiction since, by \eqref{asymptUA2}, the functions $U_A(y)$ and $U_B(y-y_0)$ intersect.
\end{proof}

\begin{proof}[Proof of Corollary~\ref{cor-nonrad}(iii)]
Let $f(u)=u^{p_S}$,
$u_0=u_0(|x|)$ be radially nonincreasing, piecewise $C^1$ and satisfying 
either \eqref{u0-supp1} or \eqref{u0-supp2},
and let $v_0\le u_0$, $v_0\not\equiv u_0$.
Assume to the contrary that $\limsup_{t\to\infty}\|v(\cdot,t)\|_\infty>0$.
Given $\delta>0$ 
small, we have $v(\cdot,\delta)<u(\cdot,\delta)$, hence the 
solutions $w_\eps$ with initial data
$w_\eps(x,\delta)=w_{\eps,\delta}(x):=\min(u(x,\delta),u(0,\delta)-\eps)$ satisfy 
$v\leq w_\eps\leq u$ for $t\ge\delta$ 
and $\eps\in(0,\eps_1]$ with $\eps_1$ small enough.
Moreover, $w_\eps\to u$ as $\eps\to0$.
In addition, assertion (ii) and our assumptions
guarantee that $w_\eps$ is bounded for any $\eps>0$
and does not converge to zero as $t\to\infty$.
Notice also that the zero number $z(w_\eps(\cdot,t)-U_A)$ in $(0,\infty)$ 
is nonincreasing and finite for $t>\delta$ and $A>0$,
since 
$z(w_\eps(\cdot,\delta)-U_A)\le z(u(\cdot,\delta)-U_A)+1<\infty$.

Similarly as in the proof of Theorem~\ref{thm-nonrad-3}(ii)
we find that for some $t_k\to\infty$ and $A_\eps>0$ we have 
$w_\eps(0,t_k)=\max w_\eps(\cdot,t_k)\to A_\eps$ and
$w_\eps(\cdot,t_k)\to U_{A_\eps}$ locally uniformly.
Assume $B>A_\eps$. 
Since the zero number $z(w_\eps(\cdot,t)-U_B)$ is nonincreasing and finite for $t>\delta$,
and it drops whenever $w_\eps(\cdot,0)=B$, we have $w_\eps(\cdot,t)<B$ for $t$ large enough.
Similarly we obtain $w_\eps(\cdot,0)>B$ for $t$ large enough provided $B<A_\eps$.
Consequently, $w_\eps(0,\cdot)\to A_\eps$ as $t\to\infty$ 
and the same arguments as in the proof of Theorem~\ref{thm-nonrad-3}(ii) 
yield $w_\eps(\cdot,t)\to U_{A_\eps}$ as $t\to\infty$ locally uniformly.
Let $\eps_2\in(0,\eps_1)$. 
We have $A_{\eps_2}=A_{\eps_1}=:A$, since otherwise $U_{A_{\eps_1}}$ and $U_{A_{\eps_2}}$ intersect
(cf.~\eqref{asymptUA2}),
which contradicts $w_{\eps_2}>w_{\eps_1}$.

Let us show that
\begin{equation} \label{zR002}
\hbox{if $M>\max(u_0(0),A)$ is large enough, then $z(u_0-U_M)=2$.}
\ee
Assume \eqref{u0-supp2} (the case of \eqref{u0-supp1} is similar).
We first claim that
\begin{equation} \label{zR01}
z_{(R_0,\infty)}(u_0-U_M)=1,
\ee
 where $z_{(R_0,\infty)}$ is the zero number in $(R_0,\infty)$
and $R_0$ is specified below.
Set $c_n:=(n(n-2))^{(n-2)/2}$. Then
$$ \begin{aligned}
  U_M(r) &=\frac{c_n}M (n(n-2)M^{-4/(n-2)}+r^2)^{-(n-2)/2},\\
  -U'_M(r) &=\frac{(n-2)c_n}M r(n(n-2)M^{-4/(n-2)}+r^2)^{-n/2},
\end{aligned} $$
hence
\begin{equation} \label{limM}
\lim_{r\to\infty}MU_M(r)r^{n-2}=-\lim_{r\to\infty}\frac M{n-2}U_M'(r)r^{n-1}= c_n,
\ \hbox{ uniformly for }M\ge1.
\end{equation}
Now \eqref{u0-supp2} and \eqref{limM} guarantee the existence of $R_0>0$ large such that
\begin{equation} \label{ifthen1}
\hbox{ if $M\ge1$, $r\ge R_0$ and $u_0(r)=U_M(r)$, then $-u_0'(r)>-U_M'(r)$.}
\end{equation}
Let $C_1,C_2,C_3>0$ be such that
\begin{equation} \label{C123}
 u_0(r)\in[C_1,C_2] \hbox{ and } |u_0'(r)|\le C_3 \hbox{ for any } r\in[0,R_0].
\end{equation}
Let $M_1\ge1$ be such that $U_M(R_0)<C_1$ for $M\ge M_1$.
Then \eqref{ifthen1}  and $\lim\limits_{R\to\infty}\frac{u_0(r)}{U_M(r)}=0$ imply
\eqref{zR01}.
We next claim that 
\begin{equation} \label{zR02}
\hbox{ if $M>C_2$ is large enough, then $z_{(0,R_0)}(u_0-U_M)=1$.}
\ee
To this end, it suffices to show that, for $M>C_2$ large enough,
\begin{equation} \label{ifthen2}
\hbox{ if $r\in(0,R_0)$ and $u_0(r)=U_M(r)$, then $-U_M'(r)>-u_0'(r)$.}
\end{equation}
Thus assuming that $u_0(r)=U_M(r)$ for some $r\in(0,R_0)$ and using \eqref{C123}, we obtain
$$ C_2\ge u_0(r)=U_M(r)=c_n(n(n-2)M^{-2/(n-2)}+M^{2/(n-2)}r^2)^{-(n-2)/2},$$
which implies
$$ M^{2/(n-2)}r^2\ge \Bigl(\frac{c_n}{C_2}\Bigr)^{2/(n-2)}-n(n-2)M^{-2/(n-2)}
   \ge \frac12 \Bigl(\frac{c_n}{C_2}\Bigr)^{2/(n-2)}=:C_4 $$
provided $M\ge M_2\ge M_1$.
This estimate, the equality $U_M(r)=u_0(r)$ and \eqref{C123} imply
$$ -U_M'(r)=\frac rn M^{2/(n-2)}(U_M(r))^{n/(n-2)}
 \ge \frac{\sqrt{C_4}}n C_1^{n/(n-2)}M^{1/(n-2)} > C_3 \ge-u_0'(r)$$
provided $M\ge M_3\ge M_2$.
This shows \eqref{zR02} which, combined with \eqref{zR01} and $u_0(R_0)\ne U_M(R_0)$, yields \eqref{zR002}.

Now fix $M>\max(u_0(0),A)$ such that $z(u_0-U_M)=2$.
Then we also have
$z(w_\eps(\cdot,\delta)-U_M)=2$ and $w_\eps(0,\delta)<M$
for $\delta$ and $\eps$ small enough.
Since $u$ is unbounded, there exists $t_M$ such that $u(0,t_M)>M$,
hence $w_\eps(0,t_M)>M$ for $\eps$ small enough, which implies
$z(w_\eps(\cdot,t_M)-U_M)\le 1$.
Since $w_\eps(0,t)\to A$ as $t\to\infty$, we have $w_\eps(0,t_1)<M$
for some $t_1>t_M$, hence $z(w_\eps(\cdot,t_1)-U_M)=0$, $w_\eps\le U_M$,
which implies $w_\eps\to0$ as $t\to\infty$ due to \cite[Theorem~1.14]{GNW92}, a contradiction.    
\end{proof}

\section{Examples, remarks and open problems} \label{sec-ex}

In Section~\ref{sec-intro} we gave numerous applications of our results for the classical case $f(u)=u^p$.
Also, as noted there, the class of nonlinearities $f$ covered by our results is very large,
and for instance contains slightly superlinear functions like
$f(u)=u\log^2(1+u)$, as well as exponentially growing functions like $f(u)=e^u-u-1$.
In this section we consider two additional classes of nonlinearities,
and then comment on various issues related with notions of (sub-/super-)threshold solutions.

\begin{example} \label{exam-log} \rm
{\bf Logarithmic nonlinearities.}
Let $f(u)=u^p\log^a(2+u^2)$, $p>1$.
This model nonlinearity without scale invariance has been considered
in the study of blow-up rate (see \cite{HZ22}) and Liouville-type theorems
(see \cite{QS25p}). 
If $p<p_S$, $\Omega=\R^n$ and
$$a\in(-(p-1)c_\ell,(p_S-p)c_\ell],\quad\hbox{where }\ 
c_\ell:=\inf_{u>0}g(u),\quad g(u):=\log(2+u^2)\frac{2+u^2}{2u^2},$$
then \cite[Example~4.2]{QS25p} guarantees
$u^{p-1}\log^a(2+u^2)(\cdot,t)\le C(n,f)/t$
for any global solution $u$, hence global solutions are
bounded and tend to zero as $t\to\infty$.
Notice also that 
$$c_\ell=\inf_{u>0}g(u)=g(u_\ell),\quad\hbox{where }\ 
 g'(u_\ell)=0,\ 2\log(2+u_\ell^2)=u_\ell^2,\ c_\ell=\frac14(2+u_\ell^2).$$
(We have $u_\ell\approx1.832$, $c_\ell\approx1.339$.)

Next consider the case $p=p_S$. 
Conditions \eqref{f1}--\eqref{f4} are satisfied if, for example, $a\le0$, and
$f'(u)u\ge(1+\eta)f(u)$ for some $\eta>0$,
which is equivalent to $a\ge-(p-1-\eta)c_\ell$.
Hence Theorem~\ref{thm-nonrad-3} can be applied
for all $a\in(a_\ell,0]$, where $a_\ell:=-(p-1)c_\ell$. 
On the other hand, one can easily show that Theorem~\ref{thm-intro} can be applied
for $a\in(a_1,\infty)$, with some $a_1>a_\ell$, but not for all $a\in(a_\ell,0]$,
because the assumption $f''\ge0$ fails for $a$ close to $a_\ell$.
Indeed, noting that
$$ f''(u)=(p-1)u^{p-2}\log^{a-2}(2+u^2)\Bigl(\frac{2u^2}{2+u^2}\Bigr)^2h(u),$$
where 
$$ h(u)=h_a(u):=pg^2(u)+\frac{a}{p-1}g(u)\frac{4p+2+(2p-1)u^2}{2+u^2}+\frac{a}{p-1}(a-1),$$
we have $h_{a_\ell}(u_\ell)=0<h_{a_\ell}'(u_\ell)$.
Therefore there exists $u\in(0,u_\ell)$ such that $h_{a_\ell}(u)<0$,
hence $f''(u)<0$ for $a>a_\ell$ close to $a_\ell$.

Theorem~\ref{thm-intro} can also be used if $p>p_S$.
On the other hand, Theorem~\ref{thm-nonrad-3} requires
the existence of global unbounded solutions and this
does not seem to be known if $a\ne0$ (cf.~also Remark~\ref{rem-GUpS}).
\qed\end{example}
 
\begin{example} \label{exam-sum} \rm
{\bf Sums of powers.}
Let $\Omega=\R^n$ and $f(u)=u^{p_S}+\eta u^p$, where $p>1+\frac2n$ and $\eta\ge0$.
Assume that, for some $R>0$,
\begin{equation} \label{vpR}
\left. \begin{aligned}
&\hbox{$\varphi:\R^n\to[0,\infty)$ is radially nonincreasing with compact support,}\\
&\hbox{$\varphi=\varphi(|x|)\in C^1([0,R])$ and $\varphi(R)=0>\varphi'(R-)$.}
\end{aligned}\ \right\} 
\end{equation}
Let $u_\lambda$ be the solution with initial data $\lambda\varphi$ 
and $u^*=u_{\lambda^*}$ be the corresponding $\eps$-threshold.
We can describe the behavior of $u^*$ as follows, according to whether $\eta=0$ or $\eta>0$ and 
$p$ is sub- or supercritical.

\vskip 3pt

$\bullet$ Let $\eta=0$ and $n\in\{3,4\}$. Then Remark~\ref{rem-decay}(i), Theorem~\ref{thm-intro}(ii)--(iii)
and Corollary~\ref{cor-nonrad} guarantee 
that 
$$\hbox{$u^*$ is 
global unbounded 
and each subthreshold converges to zero as $t\to\infty$.}$$

\vskip 3pt

$\bullet$ Let $\eta>0$ and $p<p_S$. We show that, under an additional assumption on $\varphi$, 
\begin{equation} \label{vpR2}
\hbox{$u^*$ is global bounded and $u^*\le Ct^{-1/(p-1)}$, $t>0$.}
\end{equation}
Namely, recalling that the equation $-\Delta w=f(w)$ in $\R^n$ does not possess entire positive solutions
(see \cite{QS25p} and the references therein), and choosing $M=\varphi(0)$,
we can find $\tilde R>0$ such that the corresponding Cauchy-Dirichlet problem in 
$B_{\tilde R}=\{x\in\R^n:|x|<\tilde R\}$
possesses a radial positive steady state $w_0$ with $w_0(0)=M$.
We claim that property \eqref{vpR2} holds whenever
$$\hbox{\eqref{vpR} is true with
$R\ge\tilde R$ and $0\ge \varphi'(r)\ge w_0'(r)\ \hbox{ for }\ r\in[0,\tilde R)$}$$
(this assumption guarantees $z(\lambda\varphi-w_0)=1$ for $\lambda<1$).
To see this, first observe that, since $w_0$ is a strict subsolution for the problem in $\R^n$, 
the solution $w$ of the Cauchy problem  with initial data $w_0$ 
satisfies $w(\cdot,\delta)\ge(1+\eps)w_0$ in $B_{\tilde R}$ for some $\delta,\eps>0$.
The comparison of $w$ with the solution $\tilde w$
of the Cauchy-Dirichlet problem in $B_{\tilde R}$ 
with initial data $(1+\eps)w_0$
together with \cite[Theorem 17.10]{QS19} 
then show that $w$ blows up in finite time.
Now assume for contradiction that $u^*(0,t)>M$ for some $t>0$.
Then $u_\lambda(0,t)>M$ for some $\lambda<\lambda^*$
and a zero number argument shows $u_\lambda(\cdot,t)>w_0$,
hence $u_\lambda$ blows up in finite time,
which is a contradiction.
Consequently, $u^*$ is global and bounded; $u^*\le M$.
In addition, \cite[Theorem 3.1(ii)]{QS25p} implies $(u^*)^{p-1}\le C/t$.

\vskip 3pt

$\bullet$ Let $\eta>0$ and $p>p_S$. Then
 $$\hbox{$u^*$ 
does not converge to zero}$$
(we do not know whether $u^*$ is global or not, but this already shows that its behavior is different from the case $p<p_S$).
Indeed, choosing $M>0$ small enough, the Pohozaev identity implies that 
there exists an entire positive solution $w$ of $-\Delta w=f(w)$ in $\R^n$
satisfying $w(0)=M<u_0^*(0)$ and $z(w-\lambda\varphi)=1$ for $\lambda\ge\lambda^*$.
If $u^*(0,t_0)<M$ for some $t_0$, then we also have $u_\lambda(0,t_0)<M$ for
some $\lambda>\lambda^*$, hence $u_\lambda\le w$ for $t\ge t_0$,
which is a contradiction (since $u_\lambda$ blows up).
\qed\end{example}

\begin{remark} \label{rem-pJL}  \rm 
{\bf Bounded and unbounded subthresholds.}
Assume $\Omega=\R^n$, $p\ge p_{JL}$, set $m:=2/(p-1)$
and let $U_*(x):=c_p|x|^{-m}$ denote  the singular stationary solution 
(see \cite[Remark~3.6(i)]{QS19}). 
If $u_0\le U_*$ and 
\begin{equation} \label{ell}
  \lim_{|x|\to\infty}|x|^{\ell}(U_*-u_0)(x)=0, \quad 
\ell: =\frac12(n-2-\sqrt{(n-2-2m)^2-8(n-2-m)}) >0,
\end{equation}
then the corresponding solution $u$ is a global unbounded $\eps$-threshold
but it is not a threshold 
(see \cite{PY03}, and see also \cite{FKWY06,FKWY07,FKWY11} for more precise statements
on the asymptotic behavior of such solutions).
Condition \eqref{ell} is optimal in the following sense:
If $p>p_{JL}$, $u_0\le U_*$ and 
\begin{equation} \label{belowU*}
\liminf_{|x|\to\infty}|x|^\ell(U_*(x)-u_0(x))>0,
\end{equation}
then the solution $u$ is bounded, see \cite{FWY08}.

Theorem~\ref{thm-intro} implies a similar sufficient condition for the boundedness of $u$
in a more general situation: Assume that $\Omega=\R^n$, $f$ 
satisfies \eqref{f10} and $f(u)\leq u^r$ for $u$ small and some $r\ge 1$.
Let $w$ be a global unbounded solution with initial data satisfying
$w_0(x)\le C|x|^{-\alpha}$ for some $C, \alpha>0$ and $|x|$ large.
If $u_0<w_0$ and
\begin{equation} \label{beloww0} 
\liminf_{|x|\to\infty}|x|^{\alpha r}(w_0(x)-u_0(x))>0,
\end{equation}
then the solution $u$ with initial data $u_0$ is bounded.
Notice 
that if $f(u)=u^p$ with $p>p_{JL}$ and $w_0\le CU_*$,
then $\alpha=2/(p-1)$ and $r=p$, hence $\alpha r<\ell$, 
i.e. our condition \eqref{beloww0} is stronger than \eqref{belowU*}
if $w_0\le U_*$.
However, 
our initial data $w_0,u_0$ need not lie below $U_*$,
our results also imply blow-up in finite time
if $u_0>w_0$ and
$$\liminf_{|x|\to\infty}|x|^{\alpha r}(u_0(x)-w_0(x))>0,$$
and they can also be used for more general $f$.
\qed\end{remark}

\begin{remark} \label{rem-epsf} \rm
{\bf Non-coincidence of the three notions of thresholds.}

\smallskip

(i) Let $\Omega=\R^n$, $f(u)=u^p$ and $p>1+2/n$.
Then \cite[Theorem~1]{Q17} and \cite[Theorem~6.1]{PY03}
imply the existence of $L^*>0$ with the following properties:
Assume that
$u_0(x)\leq L^*|x|^{-2/(p-1)}$ for all $x$ and $u_0(x)=L^*|x|^{-2/(p-1)}$ for $|x|$ large.
Then the solution $u$ with initial data $u_0$ is global, 
but the solution with initial data $(1+\eps)u_0$ blows up in finite time.
In addition, $u$ is unbounded if $p\ge p_{JL}$.
If $u_0$ is also radial and radially nonincreasing, then
these facts and Theorem~\ref{thm-intro} guarantee that
$u$ is an $\eps$-threshold which is not an $\eps f$-threshold if $p<n/(n-2)$
(since $f(u_0)\in L^1$ and $u$ is not a threshold),
and $u$ is an $\eps f$-threshold which is not a threshold if $p\ge p_{JL}$. 

\smallskip

(ii) 
Let $f$ satisfy \eqref{f10}
and $u_0$ satisfy $T(\lambda u_0)=\infty$ for $\lambda>0$ small.
Note that $T(\lambda u_0)<\infty$ for $\lambda>0$ large by Lemma~\ref{lem-UB}, and set
$$ \lambda^*:=\sup\{\lambda>0:T(\lambda u_0)=\infty\}.$$
Then the solution $u^*$ with initial data $\lambda^*u_0$ is an $\eps$-threshold,
hence a threshold if $\Omega$ is bounded.
If $\Omega=\R^n$ and $u_0$ is radial and radially nonincreasing,
then $u^*$ need not be even an $\eps f$-threshold, see Remark~\ref{rem-epsf}(i). 
However, choosing a continuous, radial, radially nonincreasing function
$\varphi\not\equiv0$ satisfying $\varphi\in L^1(\R^n)$,
the arguments in \cite{Q17} imply that
$u^{**}:=u^*+\mu\varphi$ is a threshold (hence $\eps f$-threshold) for suitable $\mu\ge0$.
\qed\end{remark}

\begin{remark} \label{rem-GUpS} \rm 
{\bf Unbounded global solutions in bounded domains.}
If $\Omega$ is bounded and $u$ is global unbounded, then Theorem~\ref{thm-introA} 
implies that $u$ is a threshold.
However, such solutions seem to be known only if $f(u)=u^{p_S}$
(and the proof of their existence is not easy).
In fact, the boundedness of all global solutions is known
for a large class of $f$ with subcritical growth,
and in the supercritical case,  
such boundedness follows from \cite{CDZ,BS,S17} if $f(u)=u^p$ and $\Omega$ is convex
or \cite{CFG08} if $\lim_{u\to\infty}u^{-p}f(u)=1$, $p\in(p_S,p_{JL})$,
$\Omega$ is a ball and $u$ is radial. 
It would be interesting to have examples of unbounded global solutions
for $f$ different from but ``close'' to $f(u)=u^{p_S}$. 
\qed\end{remark}

\begin{remark} \label{rem-weak} \rm 
{\bf Nonexistence of global weak continuation for superthreshold solutions.}
Assume \eqref{f10}, $\Omega$ bounded, let $u, v$ be the solutions of \eqref{MPf}, with respective initial data 
$u_0\le v_0$, $u_0\not\equiv v_0$, and assume that $u$ is a threshold
(hence $v$ blows up in finite time).
Then, as consequence of the arguments in the proof of Theorem~\ref{thm-subsol2}.
it follows that $v$ cannot be continued as a global weak solution
(cf.~\cite{NST} or \cite[Definition~48.8]{QS19} for precise definition).

To see this, assume $v_0=\lambda u_0$ for some $\lambda>1$ without loss of generality,
and suppose for contradiction that such a continuation exists,
still denoted $v$. By \cite[Proposition~3]{NST}, $v$ satisfies $L^1_\delta$ bounds, i.e.~$\sup_{t>0}\|v(t)\|_{L^1_\delta}<\infty$.
Let $w$ be the minimal solution starting from $v_0$.
By definition (cf.~\cite[Remark~15.4(vii)]{QS19}), $w$ is the increasing limit of global classical solutions $w_j$ of truncated problems,
replacing $f$ by $f_j$ in \eqref{MPf}, where $f_j$ is any sequence of $C^1$ functions such that $f_j\nearrow f$ 
locally uniformly on $[0,\infty)$
and the corresponding solutions are global.
In particular we may use the convex affine truncations $f_j(s)=f(j)+(s-j)f'(j)$ for $s>j$.
Then we have $w\le v$ by the minimality of $w$
(more precisely, this follows from \cite[Remark~15.4(vii)]{QS19}, noting that $v$ is an integral solution by \cite[Corollary~48.10]{QS19}).
Consequently, $w$ also satisfies $L^1_\delta$ bounds.
Fix $\mu \in (1,\lambda)$, denote by $z$ the solution of \eqref{MPf}, with initial data $\mu u_0$,
by $T<\infty$ its blowup time and by $z_j$ the solution of the truncated problems.
By the proof of Theorem~\ref{thm-subsol2} (applied with $f_j$), we have 
$$\frac{f_j^k(z_j)}{z_j^{k-1}}\le C_k(w_j+1) \le C_k(1+w)\quad\hbox{in $\Omega\times(0,T)$}.$$
for all $k,j\ge 1$, where the constants $C_k$ are independent of $j$.
Passing to the limit $j\to\infty$, we get $\frac{f^k(z)}{z^{k-1}}\le C_k(1+w)$, hence 
$\frac{f^k(z)}{z^{k-1}}$ satisfies $L^1_\delta$ bounds in $\Omega\times(0,T)$ for any $k\ge 1$.
We then deduce from Lemma~\ref{lem-smoothing} that $z$ is bounded in $\Omega\times(0,T)$:
a contradiction.
\qed\end{remark}

\vskip3mm

\noindent{\bf Declarations.} The authors have no competing interests to declare regarding the work in this paper.

\vskip3mm

\noindent{\bf Acknowledgement.}
The first author was supported in part by the Slovak Research and Development Agency
under the contract No.~APVV-23-0039 and by VEGA grant 1/0245/24.

\end{document}